\documentclass[12pt]{amsart}
\usepackage{amsmath,latexsym,amsfonts,amssymb,amsthm}
\usepackage{geometry}
\usepackage{hyperref}
\usepackage{mathrsfs}
\usepackage{graphicx,color}
\usepackage[alphabetic]{amsrefs}
\usepackage[all,cmtip]{xy}
\usepackage{bbm}

\newtheorem{prop}{Proposition}[section]
\newtheorem{thm}[prop]{Theorem}
\newtheorem{cor}[prop]{Corollary}
\newtheorem{conj}[prop]{Conjecture}
\newtheorem{lem}[prop]{Lemma}

\theoremstyle{definition}
\newtheorem{que}[prop]{Question}
\newtheorem{defn}[prop]{Definition}
\newtheorem{expl}[prop]{Example}
\newtheorem{rem}[prop]{\it Remark}

\newtheorem*{claim*}{Claim}
\newtheorem{lemdefn}[prop]{Lemma-Definition}
\newtheorem{case}{Case}

\newcommand{\bP}{\mathbb{P}}
\newcommand{\bC}{\mathbb{C}}
\newcommand{\bR}{\mathbb{R}}
\newcommand{\bA}{\mathbb{A}}
\newcommand{\bQ}{\mathbb{Q}}
\newcommand{\bZ}{\mathbb{Z}}
\newcommand{\bN}{\mathbb{N}}

\newcommand{\bG}{\mathbb{G}}
\newcommand{\bk}{\mathbbm{k}}

\newcommand{\tX}{\widetilde{X}}

\newcommand{\tE}{\widetilde{E}}
\newcommand{\tD}{\widetilde{D}}
\newcommand{\tDelta}{\widetilde{\Delta}}
\newcommand{\tF}{\widetilde{F}}
\newcommand{\tx}{\widetilde{x}}

\newcommand{\cX}{\mathcal{X}}

\newcommand{\cC}{\mathcal{C}}
\newcommand{\cO}{\mathcal{O}}

\newcommand{\cF}{\mathcal{F}}

\newcommand{\cD}{\mathcal{D}}
\newcommand{\cP}{\mathcal{P}}
\newcommand{\DMR}{\mathcal{DMR}}

\newcommand{\fa}{\mathfrak{a}}
\newcommand{\fb}{\mathfrak{b}}
\newcommand{\fc}{\mathfrak{c}}
\newcommand{\fm}{\mathfrak{m}}

\newcommand{\fab}{\fa_{\bullet}}
\newcommand{\fbb}{\fb_{\bullet}}
\newcommand{\fcb}{\fc_{\bullet}}

\newcommand{\bbk}{\bar{\bk}}

\newcommand{\rd}{\mathrm{d}}

\newcommand{\Spec}{\mathbf{Spec}}
\newcommand{\Proj}{\mathbf{Proj}}
\newcommand{\Supp}{\mathrm{Supp}}
\newcommand{\Hom}{\mathrm{Hom}}
\newcommand{\mult}{\mathrm{mult}}
\newcommand{\lct}{\mathrm{lct}}
\newcommand{\Ex}{\mathrm{Ex}}

\newcommand{\Aut}{\mathrm{Aut}}
\newcommand{\vol}{\mathrm{vol}}
\newcommand{\ord}{\mathrm{ord}}
\newcommand{\Val}{\mathrm{Val}}

\newcommand{\hvol}{\widehat{\rm vol}}

\newcommand{\wt}{\mathrm{wt}}
\newcommand{\Coef}{\mathrm{Coef}}
\newcommand{\Diff}{\mathrm{Diff}}
\newcommand{\mld}{\mathrm{mld}}
\newcommand{\mldk}{\mathrm{mld}^{\mathrm{K}}}

\newcommand{\an}{\mathrm{an}}
\newcommand{\Bl}{\mathrm{Bl}}

\newcommand{\unit}{(\mathrm{unit})}

\numberwithin{equation}{section}

\newcommand{\red}[1]{{\textcolor{red}{#1}}}

\title[Boundedness of singularities]{On boundedness of singularities and minimal log discrepancies of Koll\'ar components}

\author{Ziquan Zhuang}
\address{Department of Mathematics, Princeton University, Princeton, NJ 08544, USA}
\email{zzhuang@princeton.edu}

\address{Department of Mathematics, Johns Hopkins University, Baltimore, MD 21218, USA}
\email{zzhuang@jhu.edu}

\date{}

\begin{document}

\maketitle

\begin{abstract}
Recent study in K-stability suggests that klt singularities whose local volumes are bounded away from zero should be bounded up to special degeneration. We show that this is true in dimension three, or when the minimal log discrepancies of Koll\'ar components are bounded from above. We conjecture that the minimal log discrepancies of Koll\'ar components are always bounded from above, and verify it in dimension three when the local volumes are bounded away from zero. We also answer a question from \cite{HLQ-ACC} on the relation between log canonical thresholds and local volumes.
\end{abstract}

\section{Introduction}

In recent years, tremendous progress has been made towards the boundedness of Fano varieties in many different contexts, see for example \cites{HMX-ACC,Birkar-bab-1,Birkar-bab-2,Jia17}. In contrast, much less is known about the boundedness of Kawamata log terminal (klt) singularities, often viewed as the local analog of Fano varieties. Certainly, one needs to be more careful about what boundedness means in the local situation, as in general a (non-isolated) singularity can have an infinite dimensional versal deformation space. As a remedy, one considers boundedness up to special degeneration, see Definition \ref{defn:special bounded}; roughly speaking, a class of klt singularities is bounded up to special degeneration (or specially bounded) if they degenerate to a bounded family of klt singularities.

Typically, we expect a class of singularities to be specially bounded after fixing some interesting invariants. One example, studied in \cites{Moraga-exceptional-mld,HLS-epsilon-plt-blowup}, is the class of $(\varepsilon,\delta)$-lc singularities. These are $\varepsilon$-lc singularities that admit $\delta$-plt blowups (see Definition \ref{defn:kc and delta-plt blowup}), and they are known to be bounded up to special degeneration \cite{HLM-special-bounded}. Another invariant that has attracted a lot of attention is the local volume of a klt singularity, originally introduced in \cite{Li-normalized-volume} in the context of K-stability. It is expected  that these two types of invariants are closely related: a positive lower bound on the local volumes would force the singularities to be $(\varepsilon,\delta)$-lc for some fixed $\varepsilon,\delta>0$ and in particular they should be bounded up to special degeneration. Since the minimal log discrepancy is always bounded from below by the local volume up to a positive dimensional constant \cite{LLX20}*{Theorem 6.13}, the main question is the existence of $\delta$-plt blowup. The following is a precise formulation of the conjecture. 

\begin{conj}[\cite{HLQ-ACC}*{Conjectures 1.6 and 8.9}] \label{conj:vol>epsilon implies bounded}
Let $n\in\bN^*$ and let $\varepsilon,\eta>0$. Then there exists some $\delta>0$ depending only on $n,\varepsilon,\eta$ such that: if $x\in (X,\Delta=\sum_{i=1}^m a_i \Delta_i)$ is an $n$-dimensional klt singularity such that
\begin{enumerate}
    \item $a_i\ge \eta$ for all $i$,
    \item each $\Delta_i$ is an effective Weil divisor, and
    \item $\hvol(x,X,\Delta)\ge \varepsilon$,
\end{enumerate}
then $x\in (X,\Delta)$ admits a $\delta$-plt blowup. Moreover, the set of such singularities are log bounded up to special degeneration.
\end{conj}

The special boundedness part of the above conjecture is also supported by the Stable Degeneration Conjecture \cites{Li-normalized-volume,LX18} from the local K-stability theory of klt singularities. It predicts that every klt singularity admits a volume preserving special degeneration to a K-semistable log Fano cone singularity, and differential geometric considerations (see \cite{SS-two-quadric}) suggest that the set of K-semistable log Fano cone singularities may be bounded when the local volume is bounded from below. The recent proof \cite{LXZ-HRFG} of the Higher Rank Finite Generation Conjecture further suggests that Conjecture \ref{conj:vol>epsilon implies bounded} will be a crucial ingredient toward the proof of the Stable Degeneration Conjecture\footnote{Postscript note: the local version of the Higher Rank Finite Generation Conjecture is later settled in \cite{XZ-SDC}. Combined with earlier works \cites{Blu18,LX18,LwX-polystable-degeneration,Xu-quasi-monomial,XZ-uniqueness}, the proof of the Stable Degeneration Conjecture is now complete.}. Despite these heuristics, Conjecture \ref{conj:vol>epsilon implies bounded} is only known in dimension two, in some special cases in dimension three, and when $x\in X$ is already analytically bounded \cite{HLQ-ACC}*{Theorems 1.7 and 1.8}. 

The first result of this paper is the solution of Conjecture \ref{conj:vol>epsilon implies bounded} in dimension three.

\begin{thm}[=Corollary \ref{cor:special bdd dim=3}] \label{main:special bdd dim=3}
Conjecture \ref{conj:vol>epsilon implies bounded} holds in dimension three.
\end{thm}

In the course of the proof, we also discover the following statement.

\begin{thm}[=Theorem \ref{thm:lct>=vol}] \label{main:lct>=vol}
For any $n\in\bN^*$, there exists some constant $c=c(n)>0$ depending only on $n$ such that for any $n$-dimensional $\bQ$-Gorenstein klt singularity $x\in (X,\Delta)$ we have
\[
\lct_x(X,\Delta;\Delta)\ge c(n)\cdot \hvol(x,X,\Delta).
\]
\end{thm}

This gives a positive answer to \cite{HLQ-ACC}*{Question 8.1} as well as further evidence for Conjecture \ref{conj:vol>epsilon implies bounded}. Indeed, if we allow the above constant $c>0$ to also rely on the local volume, then the statement is essentially implied by Conjecture \ref{conj:vol>epsilon implies bounded} and the lower semi-continuity of log canonical thresholds.

Our approach to Conjecture \ref{conj:vol>epsilon implies bounded} is through another invariant of the singularity: the minimal log discrepancy of Koll\'ar components.

\begin{defn}
Let $(X,\Delta)$ be a klt pair and let $\eta\in X$ (not necessarily a closed point). The minimal log discrepancy of Koll\'ar components, denoted by $\mldk (\eta,X,\Delta)$, is the smallest log discrepancy $A_{X,\Delta}(E)$ among all Koll\'ar components $E$ over $\eta\in (X,\Delta)$.
\end{defn}

It is not hard to see (Lemma \ref{lem:mldk bdd if delta-plt blowup exist}) that in order for a class of singularities to admit $\delta$-plt blowups, their $\mldk$ are necessarily bounded from above. Our next result is the following converse.

\begin{thm}[=Theorem \ref{thm:bdd via vol and mldK}] \label{main:bdd via vol and mldK}
Let $n$ be a positive integer, let  $\varepsilon,A>0$, and let $I\subseteq [0,1]\cap \bQ$ be a finite set. Then there exists some constant $\delta=\delta(n,\varepsilon,A,I)>0$ such that any $n$-dimensional klt singularity $x\in (X,\Delta)$ with $\hvol(x,X,\Delta)\ge \varepsilon$, $\mldk(x,X,\Delta)\le A$ and $\Coef(\Delta)\subseteq I$ admits a $\delta$-plt blowup.
\end{thm}

As a direct application, we also obtain the toric case of Conjecture \ref{conj:vol>epsilon implies bounded}, proven independently by \cite{MS-toric} using convex geometry.

\begin{cor}[=Proposition \ref{prop:toric bdd}]
Let $n\in \bN^*$ and let $\varepsilon>0$. Then there are only finitely many toric singularities $x\in X$ $($up to isomorphism$)$ that supports a klt singularity $x\in (X,\Delta)$ with $\hvol(x,X,\Delta)\ge \varepsilon$ $($for some effective $\bQ$-divisor $\Delta)$.
\end{cor}

If we compare Theorem \ref{main:bdd via vol and mldK} with Conjecture \ref{conj:vol>epsilon implies bounded}, it is natural to expect that the assumption on $\mldk(x,X,\Delta)$ in Theorem \ref{main:bdd via vol and mldK} is redundant. This leads us to make the following conjecture.

\begin{conj}[BDD for $\mldk$ when $\hvol\ge \varepsilon$] \label{conj:bdd for mldK, vol>epsilon}
Let $n\in\bN^*$, $\varepsilon>0$ and let $I\subseteq [0,1]\cap \bQ$ be a finite set. Then there exists some $A=A(n,\varepsilon,I)$ such that 
\[
\mldk(\eta,X,\Delta)\le A
\]
for any $n$-dimensional klt pair $(X,\Delta)$ with $\Coef(\Delta)\subseteq I$ and any $($not necessarily closed point$)$ $\eta\in X$ with $\hvol(\eta,X,\Delta)\ge \varepsilon$.
\end{conj}

We will show that Conjecture \ref{conj:bdd for mldK, vol>epsilon} holds in dimensions up to three (Corollary \ref{cor:bdd for mld^K, vol>epsilon, dim=3}), and in any dimension it implies Conjecture \ref{conj:vol>epsilon implies bounded} (Theorem \ref{thm:delta-plt blowup strong version}). Together with Theorem \ref{main:bdd via vol and mldK} they constitute the three major steps in the proof of Theorem \ref{main:special bdd dim=3}.

Shokurov has conjectured that the set of minimal log discrepancies (mld) satisfies the ascending chain condition (ACC) \cite{Sho-mld-conj}. In particular, there should be an upper bound on the mlds that only depends on the dimension. This is known as the boundedness (BDD) conjecture for mld. On the other hand, in the local study of klt singularities it is often more natural to consider Koll\'ar components rather than general exceptional divisors. Therefore, we are also tempted to propose the following stronger conjecture, although we are only able to partially verify it at codimension two points (Proposition \ref{prop:bdd mldk, codim=2}).

\begin{conj}[BDD for $\mldk$] \label{conj:bdd for mldK}
Let $n\in\bN^*$ and let $I\subseteq [0,1]$ be a DCC set. Then there exists some constant $A=A(n,I)$ depending only on $n$ and $I$ such that 
\[
\mldk(\eta,X,\Delta)\le A
\]
for any klt pair $(X,\Delta)$ with $\Coef(\Delta)\subseteq I$ and any point $\eta\in X$ of codimension $n$.
\end{conj}

\subsection{Outline of proofs}

One of the technical steps in proving Theorem \ref{main:special bdd dim=3} is to verify Conjecture \ref{conj:bdd for mldK, vol>epsilon} in dimension three. While it is well known that $\mld(x,X,\Delta)\le 3$ for any $3$-dimensional klt singularity $x\in (X,\Delta)$, the divisors that compute the mld are usually not Koll\'ar components, thus bounding $\mldk$ from above presents a very different problem. In fact, as Example \ref{ex:mld^K depend on coef set} shows, already on smooth surfaces $\mldk$ can be arbitrarily large if we allow the boundary $\Delta$ to vary. 

Certainly, the advantage of working in dimension three is that we have a classification result for terminal singularities \cite{Mori-terminal-classification}. Thus we need to find ways to reduce to the terminal case and to take care of the additional boundary. The first observation is that every Koll\'ar component is an lc place of a bounded complement \cite{Birkar-bab-1}. If we take a bounded complement $D$ of a klt singularity $x\in X$ and a terminal modification $\pi\colon Y\to X$, the lc places of $(X,D)$ are the same as lc places of $(Y,\pi^*D)$, so we may hope to find Koll\'ar components over $X$ with bounded log discrepancy by taking suitable Koll\'ar components over $Y$ that are lc places of $(Y,\pi^*D)$. In general, Koll\'ar components over a birational model do not descend to Koll\'ar components over the singularity, but this is the case if they are lc places of \emph{special} complements, a notion that first appears in the recent proof of the higher rank finite generation conjecture \cite{LXZ-HRFG}. Therefore, a key step in our proof of Conjecture \ref{conj:bdd for mldK, vol>epsilon} is the construction of \emph{bounded special complements}. This step works in any dimension and allows us to reduce Conjecture \ref{conj:bdd for mldK, vol>epsilon} to the terminal singularity case (with boundary). Another technical observation (see Section \ref{sect:reduce to mld^lc and 1-comp}) involving special complements further reduces the question to the case where the terminal pair $(X,\Delta)$ has a reduced complement. From there we are able to use classification results to conclude Conjecture \ref{conj:bdd for mldK, vol>epsilon} in dimension three.

To achieve the bounded part of the special complements, we need uniform control of various invariants of the singularity, and this is where Theorem \ref{main:lct>=vol} plays a crucial role. The key to the proof of Theorem \ref{main:lct>=vol} is a uniform Izumi type estimate. Recall that the usual Izumi type inequality states that for any klt singularity $x\in (X,\Delta)$, there exists some constant $c>0$ such that
\[
\lct_x(X,\Delta;D)\ge  \frac{c}{\mult_x D}
\]
for any effective $\bQ$-Cartier divisor $D$ on $X$ (see e.g. \cite{Li-normalized-volume}*{Theorem 3.1}). However, the constant $c$ in general depends on the singularity $x\in (X,\Delta)$. Our discovery is that a uniform constant can be achieved if we replace $\mult_x$ by the minimizing valuation of the normalized volume function (Corollary \ref{cor:Izumi via vol minimizer}). In particular, we can compare both sides of the inequality in Theorem \ref{main:lct>=vol} through the minimizing valuation.

Finally we make a few remarks on the proof of Theorem \ref{main:bdd via vol and mldK}. The idea is to show that the Koll\'ar components that have bounded log discrepancy belong to a bounded family of log Fano pairs, which is the case if the Cartier indices are bounded on the corresponding plt blowup. Using the finite degree formula proved in \cite{XZ-uniqueness}, this boils down to a few estimates of local volumes on the plt blowup, see Section \ref{sect:bdd via mld^K}.

\subsection*{Acknowledgement}

The author is partially supported by the NSF Grant DMS-2240926 (formerly DMS-2055531) and a Clay research fellowship. He would like to thank Yuchen Liu and Chenyang Xu for helpful discussions, as well as Jingjun Han and Joaqu\'in Moraga for comments. He is also grateful to the anonymous referees for many valuable suggestions.

\section{Preliminaries}

\subsection{Notation and conventions}

Throughout this paper, we work over an algebraically closed field of characteristic $0$. We follow the standard terminology from \cites{KM98,Kol13}.

A singularity $x\in (X,\Delta)$ consists of a pair $(X,\Delta)$ (i.e. a normal variety $X$ together with an effective $\bQ$-divisor $\Delta$) and a closed point $x\in X$. We will always assume that $X$ is affine and $x\in \Supp(\Delta)$ (whenever $\Delta\neq 0$). In general, when we discuss local properties of a pair $(X,\Delta)$ at a (not necessarily closed) point $\eta$, we will freely shrink $X$ around $\eta$.

Suppose that $X$ is a normal variety. A prime divisor $F$ on some birational model $\pi\colon Y\to X$ (where $Y$ is normal and $\pi$ is proper) of $X$ is called a \emph{divisor over} $X$. Its center, denoted $C_X(F)$, is the generic point of its image in $X$. 

A \emph{valuation} $v$ on $X$ is an $\bR$-valued valuation $v: K(X)^{\times}\to \bR$ (where $K(X)$ denotes the function field of $X$) such that $v$ has a center on $X$ and $v|_{k^\times}=0$. The set of valuations on $X$ is denoted as $\Val_X$. 

For a valuation $v$ on $X$ and $m\in\bN$, its \emph{valuation ideal sheaf} $\fa_m(v)$ is
\[
\fa_m(v):=\{f\in \cO_X\mid v(f)\ge m\}. 
\]

When we refer to a constant $A$ as $A=A(n,\varepsilon,\cdots)$ we mean it only depends on $n,\varepsilon,\cdots$.

\subsection{Koll\'ar components}

\begin{defn}
Let $(X,\Delta)$ be a sub-pair (i.e. $\Delta$ need not be effective) and let $F$ be a divisor over $X$. We will say $F$ is a divisor over $\eta\in (X,\Delta)$ if $\eta$ is the center of $F$. When $F$ is a divisor on $X$ we write $\Delta=\Delta_1+aF$ where $F\not\subseteq \Supp(\Delta_1)$; otherwise let $\Delta_1=\Delta$.
\begin{enumerate}
	\item $F$ is said to be primitive over $X$ if there exists a projective birational morphism $\pi:Y\to X$ such that $Y$ is normal, $F$ is a prime divisor on $Y$ and $-F$ is a $\pi$-ample $\bQ$-Cartier divisor. We call $\pi:Y\to X$ the associated prime blowup (it is uniquely determined by $F$).
	\item $F$ is said to be of plt (resp. lc) type over $(X,\Delta)$ if it is primitive over $X$ and the pair $(Y,\Delta_Y+F)$ is plt (resp. lc) in a neighbourhood of $F$, where $\pi:Y\to X$ is the associated prime blowup and $\Delta_Y$ is the strict transform of $\Delta_1$ on $Y$. When $(X,\Delta)$ is klt (resp. lc) and $F$ is exceptional over $X$, $\pi$ is called a plt (resp. lc) blowup over $X$. If in addition $(Y,\Delta_Y+F)$ is $\delta$-plt in a neighbourhood of $F$ for some $\delta>0$, we say that $\pi$ is a $\delta$-plt blowup.
\end{enumerate}
\end{defn}

The following result from \cite{BCHM} will be frequently used.

\begin{lem} \label{lem:BCHM extract divisor}
Let $(X,\Delta)$ be a klt pair and let $E$ be a divisor over $X$. Assume that there exists an effective $\bQ$-Cartier $\bQ$-divisor $D$ such that $(X,\Delta+D)$ is lc and $A_{X,\Delta+D}(E)=0$. Then $E$ is of lc type. 
\end{lem}

\begin{proof}
Since $(X,\Delta+(1-\varepsilon)D)$ is klt and $0<A_{X,\Delta+(1-\varepsilon)D}(E)\ll 1$ when $0<\varepsilon\ll 1$, we know that $E$ is primitive by \cite{BCHM}*{Corollary 1.4.3}  and the basepoint free theorem. Let $\pi\colon Y\to X$ be the associated prime blowup. By assumption we also have $\pi^*(K_X+\Delta+D)\ge K_Y+\Delta_Y+E$. Since $(X,\Delta+D)$ is lc, it follows that $(Y,\Delta_Y+E)$ is lc, i.e. $E$ is of lc type.
\end{proof}

\begin{defn}[Koll\'ar Components] \label{defn:kc and delta-plt blowup}
Let $(X,\Delta)$ be a klt pair and let $\eta\in X$ (not necessarily a closed point). A divisor over $\eta\in (X,\Delta)$ is called a {\it Koll\'ar component} over $\eta\in(X,\Delta)$ if it's of plt type over $(U,\Delta|_U)$ for some neighbourhood $U\subseteq X$ of $\eta$. We say that $\eta\in (X,\Delta)$ admits a $\delta$-plt blowup (for some $\delta>0$) if it has a Koll\'ar component whose associated prime blowup is a $\delta$-plt blowup.
\end{defn}

By \cite{Xu-pi_1-finite}, every klt singularity has a Koll\'ar component. More precisely we have:

\begin{lem} \label{lem:kc exist on every dlt model}
Let $(X,\Delta)$ be a klt pair and let $D$ be an effective $\bQ$-Cartier $\bQ$-divisor such that $(X,\Delta+D)$ is strictly lc. Let $\pi\colon (Y,\Gamma)\to (X,\Delta)$ be a dlt modification, where $K_Y+\Gamma=\pi^*(K_X+\Delta+D)$. Then at least one component of $\lfloor \Gamma \rfloor$ is of plt type over $(X,\Delta)$.
\end{lem}

\begin{proof}
This follows from the proof of \cite{Xu-pi_1-finite}*{Lemma 1}, c.f. \cite{LX-stability-kc}*{Proposition 2.10}. The point is that if $W\to X$ is a log resolution of $(X,\Delta+D)$ then some prime divisor on $W$ that is an lc place of $(X,\Delta+D)$ will be of plt type over $(X,\Delta)$. The resolution can be chosen as a blowup of $Y$ along non-SNC locus of $(Y,\Gamma)$, and the dlt assumption ensures that none of the exceptional divisors are lc places of $(Y,\Gamma)$. Thus the plt type divisor on $W$ has to be a component of $\lfloor \Gamma \rfloor$.
\end{proof}

We next describe a criterion which will be used to verify certain weighted blowups on hypersurfaces singularities provide Koll\'ar components. 

\begin{lem} \label{lem:wt blowup kc criterion}
Let $X=(f(x_1,\cdots,x_n)=0)\subseteq \bA^n$ be a hypersurface singularity. Let $E$ be the exceptional divisor of the weighted blowup with $\wt(x_i)=a_i>0$. Assume that the hypersurface $X_0=(\mathrm{in}(f)=0)\subseteq \bA^n$ has only canonical singularity at the origin. Then $E$ is a Koll\'ar component over $0\in X$.
\end{lem}

Here $\mathrm{in}(f)$ denotes the sum of the monomials in $f$ with lowest weights. Note that for non-degenerate hypersurfaces this is proved in \cite{IP-hypsurf-exceptional}*{Proposition 3.3}. We essentially follow the same argument.

\begin{proof}
Let $f_0=\mathrm{in}(f)$ and let $\pi\colon Y\to X$ be the weighted blowup. Then $E\cong (f_0=0)\subseteq \bP(a_1,\cdots,a_n)$. Since $f_0$ is irreducible by assumption, we see that $E$ is a primitive divisor on $Y$. It remains to show that $(Y,E)$ is plt. By inversion of adjunction, this is equivalent to showing that $(E,\Diff_E(0))$ is klt. Note that $X_0$ admits a good $\bG_m$-action with $t\cdot x_i=t^{a_i}x_i$, hence $X_0\setminus \{0\}$ is a Seifert $\bG_m$-bundle in the sense of \cite{Kol-Seifert-bundle}. A direct calculation shows that its orbifold base is exactly $(E,\Diff_E(0))$. Since $X_0$ has canonical singularities (hence is klt) in a neighbourhood of the origin, using the $\bG_m$-action we see that $X_0$ is klt everywhere. Hence by \cite{Kol-Seifert-bundle}*{Proposition 19}, the local orbifold covers of $(E,\Diff_E(0))$ are all klt, hence $(E,\Diff_E(0))$ is klt as well. This finishes the proof.
\end{proof}

\subsection{Local volumes}

Given a pair $(X,\Delta)$, the log discrepancy function
\[
A_{X,\Delta}\colon \Val_X\to \bR \cup\{+\infty\},
\]
is defined as in \cite{JM12} and \cite{BdFFU15}*{Theorem 3.1}. It is possible that $A_{X,\Delta}(v) = +\infty$ for some $v\in \Val_X$, see e.g. \cite{JM12}*{Remark 5.12}. For a closed point $x\in X$, we denote by $\Val^*_{X,x}$ the set of valuations $v\in\Val_X$ with center $x$ and  $A_{X,\Delta}(v)<+\infty$.

\begin{defn}
Let $X$ be an $n$-dimensional normal variety and let $x\in X$ be a closed point. The \emph{volume} of a valuation $v\in\Val^*_{X,x}$ is defined as
\[
\vol(v)=\vol_{X,x}(v)=\limsup_{m\to\infty}\frac{\ell(\cO_{X,x}/\fa_m(v))}{m^n/n!}.
\]
Thanks to \cites{ELS03, LM09, Cut13}, the above limsup is actually a limit.
\end{defn}

We now briefly recall the definition of the volumes of klt singularities \cite{Li-normalized-volume}. 

\begin{defn} \label{defn:local volume}
Let $x\in (X,\Delta)$ be an $n$-dimensional klt singularity. For any $v\in \Val^*_{X,x}$, we define the \emph{normalized volume} of $v$ as
\[
\hvol_{X,\Delta}(v):=A_{X,\Delta}(v)^n\cdot\vol_{X,x}(v).
\]
The \emph{local volume} of $x\in (X,\Delta)$ is defined as
\[
  \hvol(x,X,\Delta):=\inf_{v\in\Val^*_{X,x}} \hvol_{X,\Delta}(v).
\]
\end{defn}

By \cites{Blu18}, the above infimum is a minimum. In fact, the minimizer is unique up to rescaling \cite{XZ-uniqueness}, and is a quasi-monomial valuation \cite{Xu-quasi-monomial}, but we do not need these facts in the sequel. 

Since we also study the singularities at non-closed points, we extend the above definition to those cases as follows.

\begin{defn}
Let $(X,\Delta)$ be a klt pair and let $\eta$ be the generic point of a subvariety $W\subseteq X$. The local volume of $(X,\Delta)$ at $\eta$ is defined to be
\[
\hvol(\eta,X,\Delta):=\hvol(w,X,\Delta)
\]
where $w\in W$ is a general closed point. This is well-defined since the local volume function $x\mapsto \hvol(x,X,\Delta)$ is constructible \cite{Xu-quasi-monomial}.
\end{defn}

We will frequently use the fact that the volume function $x\mapsto \hvol(x,X,\Delta)$ is lower semi-continuous \cite{BL-vol-lsc}. We recall a few more useful properties of local volumes.

\begin{lem} \label{lem:index bound via volume}
Let $x\in(X,\Delta)$ be a klt singularity of dimension $n$ and let $D$ be a $\bQ$-Cartier Weil divisor on $X$. Then the Cartier index of $D$ is at most $\frac{n^n}{\hvol(x,X,\Delta)}$.
\end{lem}

\begin{proof}
This follows directly from \cite{XZ-uniqueness}*{Corollary 1.4}.
\end{proof}

\begin{lem} \label{lem:vol under birational map}
Let $\pi\colon (Y,\Delta_Y)\to (X,\Delta)$ be a proper birational morphism between klt pairs. Assume that $K_Y+\Delta_Y\le \pi^*(K_X+\Delta)$. Then $\hvol(y,Y,\Delta_Y)\ge \hvol(x,X,\Delta)$ for any $x\in X$ and any $y\in \pi^{-1}(x)$.
\end{lem}

\begin{proof}
This follows from the same proof of \cite{LX-cubic-3fold}*{Lemma 2.9(2)}, which tackles the boundary-free case.
\end{proof}

\begin{lem}[\cite{HLQ-ACC}*{Lemma 2.16}] \label{lem:vol inequality}
Let $x\in (X,\Delta)$ be a klt singularity of dimension $n$, let $\lambda>0$, and let $D$ be an effective $\bQ$-Cartier divisor on $X$ such that $(X,\Delta+(1+\lambda)D)$ is lc. Then
\[
\hvol(x,X,\Delta+D)\ge \left(\frac{\lambda}{1+\lambda} \right)^n \hvol(x,X,\Delta).
\] 
\end{lem}

\begin{proof}
By assumption, for any $v\in \Val^*_{X,x}$, we have $A_{X,\Delta}(v)\ge (1+\lambda)v(D)$, thus
\[
A_{X,\Delta+D}(v)=A_{X,\Delta}(v)-v(D)\ge \left(1-\frac{1}{1+\lambda}\right) A_{X,\Delta}(v)
\]
and hence
\[
A_{X,\Delta+D}(v)^n\cdot \vol(v)\ge \left(\frac{\lambda}{1+\lambda} \right)^n A_{X,\Delta}(v)^n\cdot \vol(v).
\]
Taking the infimum over all $v\in \Val^*_{X,x}$, the lemma follows.
\end{proof}

\subsection{Complements}

In this subsection we recall some results about complements of singularities. A DCC set is a subset of $\bR$ that satisfies the descending chain condition.

\begin{defn}
A $\bQ$-complement of an lc pair $(X,\Delta)$ is an effective $\bQ$-divisor $D\sim_\bQ -(K_X+\Delta)$ such that $(X,\Delta+D)$ is lc. A \emph{$\bQ$-complement} of $\eta\in (X,\Delta)$ is an effective $\bQ$-Cartier $\bQ$-divisor $D$ such that $(X,\Delta+D)$ is lc at $\eta$ and has $\eta$ as the generic point of an lc center\footnote{Morally speaking, this should be called a strictly lc $\bQ$-complement since we also require that the complement has an lc center at $\eta$. For simplicity, we drop the phrase ``strictly lc'' when we talk about complement if the center $\eta$ is specified.}. In either case, a $\bQ$-complement $D$ is called an \emph{$N$-complement} for $N\in\bN^*$ if $N(K_X+\Delta+D)\sim 0$.
\end{defn}


\begin{lem} \label{lem:strict N-complement}
Let $n\in\bN^*$ and let $I=\bar{I}\subseteq [0,1]\cap \bQ$ be a DCC set. Then there exists some integer $N>0$ depending only on $n$ and $I$ such that if $(X,\Delta)$ is a klt pair of dimension $n$ with $\Coef(\Delta)\subseteq I$ and $\eta\in X$, then any Koll\'ar component $E$ over $\eta\in (X,\Delta)$ is an lc place of some $N$-complement. In particular, every such $\eta\in (X,\Delta)$ has an $N$-complement.
\end{lem}

\begin{proof}
This should be well known to experts but we provide a proof for the readers' convenience. We may assume that $1\in I$. Let $\pi\colon Y\to X$ be the plt blowup that extracts $E$. Since $E$ is a Koll\'ar component, $(Y,\Delta_Y+E)$ is plt, $-(K_Y+\Delta_Y+E)$ is $\pi$-ample and hence $Y$ is of Fano type over $X$. By \cite{HLS-epsilon-plt-blowup}*{Theorem 1.10} (which builds on \cite{Birkar-bab-1}*{Theorem 1.8}), after possibly replacing $X$ by a neighbourhood of $\eta$, there exists an integer $N>0$ that only depends on $n$ and $I$, and an effective $\bQ$-Cartier $\bQ$-divisor $D_Y$ on $Y$ such that $N(\Delta_Y+D_Y)$ has integer coefficients, $(Y,\Delta_Y+E+D_Y)$ is lc, and
\[
N(K_Y+\Delta_Y+E+D_Y)\sim 0.
\]
It follows that if we let $D=\pi_* D_Y$, then $K_Y+\Delta_Y+E+D_Y=\pi^*(K_X+\Delta+D)$ and hence $(X,\Delta+D)$ is lc with $E$ as an lc place. Since $\eta$ is the center of $E$, this implies that $D$ is a $\bQ$-complement of $\eta\in (X,\Delta)$. Moreover, the line bundle $N(K_Y+\Delta_Y+E+D_Y)$ descends to the line bundle $N(K_X+\Delta+D)$ by Shokurov's basepoint-free theorem (see e.g. \cite{KM98}*{Theorem 3.3}). Thus $D$ is also an $N$-complement. Since every klt singularity has a Koll\'ar component, this finishes the proof.
\end{proof}

\begin{lem} \label{lem:strict comp quasi-etale pullback}
Let $f\colon \big(y\in (Y,\Delta_Y)\big)\to \big(x\in(X,\Delta)\big)$ be a finite morphism between lc singularities such that $f^*(K_X+\Delta)=K_Y+\Delta_Y$. Let $D$ be a divisor on $X$. Then $D$ is a $\bQ$-complement of $x\in (X,\Delta)$ if and only if $f^*D$ is a $\bQ$-complement of $y\in (Y,\Delta_Y)$. 
\end{lem}

\begin{proof}
This is a direct consequence of the proof of \cite{KM98}*{Proposition 5.20(2)}.
\end{proof}

\subsection{Index one covers} \label{sect:index one covers}

The index one cover of a $\bQ$-Cartier Weil divisor $D\subseteq X$  \cite{KM98}*{Definition 5.19} is a cyclic cover $\tX=\Spec\left( \bigoplus_{0\le m\le r-1} \cO_X(mD)\right)$, where $r$ is the Cartier index of $D$. It has the property that the preimage of $D$ becomes Cartier. We will need a similar construction for multiple divisors. 

\begin{lemdefn}
Let $x\in X$ be a normal singularity and let $D_1,\cdots,D_m$ be $\bQ$-Cartier Weil divisors on $X$. Let $H$ be the subgroup of the local class group of $x\in X$ generated by all the $D_i$'s and consider
\[
\tX:=\Spec \bigoplus_{D\in H} \cO_X(D).
\]
It comes with a natural quasi-\'etale Galois morphism $\pi\colon \tX\to X$ with abelian covering group $\widehat{H}:=\Hom(H,\bC^*)$. The preimage $\pi^{-1}(x)$ consists of a single point $\tx$ and $\pi^*D_i$ is Cartier for all $1\le i\le m$. We call $\tx\in \tX$ the (simultaneous) index one cover of $D_1,\cdots,D_m$.
\end{lemdefn}

\begin{proof}
The only non-trivial statement is that $\pi^{-1}(x)$ consists of a single point $\tx$ and $\pi^*D_i$ is Cartier. The first claim can be proved as in \cite{Kol13}*{2.48(1)}: the evaluation map $\cO_X(D)\otimes \cO_X(-D)\to \cO_X/\fm$ is zero for all $D\in H\setminus\{0\}$, thus every $f\in \bigoplus_{D\in H\setminus\{0\}} \cO_X(D)$ is nilpotent in $\cO_{\tX}/\fm\cO_{\tX}$ and this implies that the preimage of $x$ is a single point in $\tX$. To see the other claim, note that the map $\pi\colon \tX\to X$ factors through the index one cover of $D_i$, thus $\pi^*D_i$ is Cartier.
\end{proof}

\subsection{Family of singularities and special boundedness}

In this subsection, we recall the definition for special boundedness of singularities, following \cite{HLQ-ACC}.

\begin{defn}
We call $B\subseteq (\cX,\cD)\to B$ a $\bQ$-Gorenstein family of klt singularities (over a normal but possibly disconnected base $B$) if
\begin{enumerate}
    \item $\cX$ is flat over $B$, and $B\subseteq \cX$ is a section of the projection,
    \item For any closed point $b\in B$, $\cX_b$ is connected, normal and is not contained in $\Supp(\cD)$,
    \item $K_{\cX/B}+\cD$ is $\bQ$-Cartier and $b\in (\cX_b,\cD_b)$ is a klt singularity for any $b\in B$.
\end{enumerate}
Given a $\bQ$-Gorenstein family $B\subseteq \cX\to B$ of klt singularities and a klt singularity $x\in X$, we denote by $(x\in X^{\an})\in (B\subseteq \cX^{\an}\to B)$, if $\widehat{\cO_{X,x}}\cong \widehat{\cO_{\cX_b,b}}$ for some closed point $b\in B$.
\end{defn}

\begin{defn}
A special test configuration of a klt singularity $x\in (X,\Delta)$ is a $\bQ$-Gorenstein family $\bA^1\subseteq (\cX,\cD)\to \bA^1$ of klt singularities over $\bA^1$, together with a $\bG_m$-action on $(\cX,\cD)$ that commutes with the standard $\bG_m$-action on $\bA^1$, such that $\left(t\in (\cX_t,\cD_t)\right)\cong \left( x\in (X,\Delta)\right)$ for all $t\in \bA^1\setminus \{0\}$. Its central fiber $0\in (\cX_0,\cD_0)$ is called a \emph{special degeneration} of $x\in (X,\Delta)$.  
\end{defn}

\begin{defn} \label{defn:special bounded}
A set $\cP$ of klt singularities is said to be \emph{log bounded up to special degeneration}
if there is a log bounded set $\cC$ of pairs, such that the following holds.

For any klt singularity $x\in (X,\Delta)$ in $\cP$, there exist a special degeneration $x_0\in (X_0,\Delta_0)$ of $x\in (X,\Delta)$,
a pair $(Y,D)\in \cC$ and a closed point $y\in Y$, such that $\left(y\in (Y,\Supp(D))\right) \cong \left(x_0\in (X_0,\Supp(\Delta_0))\right)$ in some neighborhoods of $y\in Y$ and $x_0\in X_0$ respectively.
\end{defn}

When $\cP$ is log bounded and the coefficients set $I\subseteq \bQ$ is finite, we will simply say that $\cP$ is bounded up to special degeneration, since in this case there is a $\bQ$-Gorenstein family of klt singularities such every $x\in (X,\Delta)$ in $\cP$ specially degenerates to at least one of them.

The following result is useful when showing that a class of klt singularities is log bounded up to special degeneration.

\begin{lem} \label{lem:delta-plt implies special bounded}
Let $n\in\bN^*$ and let $\varepsilon,\delta,c>0$. Then the set of $n$-dimensional $\varepsilon$-lc singularities $x\in (X,\Delta)$ with $\Coef(\Delta)\ge c$ that admits a $\delta$-plt blowup is log bounded up to special degeneration.
\end{lem}

\begin{proof}
This is essentially \cite{HLM-special-bounded}*{Theorem 4.1}, at least when $X$ is $\bQ$-Gorenstein. In general, let $\pi\colon Y\to X$ be a $\delta$-plt that extracts a Koll\'ar component $E$, and let $L=-E|_E$ be the $\bQ$-divisor defined by \cite{HLS-epsilon-plt-blowup}*{Definition A.4}. Then there is a special degeneration of $x\in (X,\Delta)$ to the orbifold cone over $(E,\Delta_E:=\Diff_E(\Delta_Y))$ with polarization $L$ (see e.g. \cite{LX-stability-kc}*{Section 2.4} or \cite{LZ-Tian-sharp}*{Proposition 2.10}). Thus it suffices to show that the triple $(E,\Supp(\Delta_E),L)$ is bounded. 

Since $(Y,E)$ is $\delta$-plt, there exists some integer $m$ depending only on $\delta$ such that $mE$ is Cartier away from a codimension two set in $E$ (it suffices to inspect the codimension one points of $E$ as in the proof of \cite{HLM-special-bounded}*{Theorem 4.1}). Thus $mL$ has integer coefficients. By adjunction, we also see that $(E,\Delta_E)$ is $\delta$-klt and log Fano, and $\Coef(\Delta_E)\ge \frac{c}{m}$. By \cite{Birkar-bab-2}*{Theorem 1.1}, we first deduce that $E$ belongs to a bounded family. But as $\Coef(\Delta_E)\ge \frac{c}{m}$ and $-K_E-\Delta_E$ is ample, we further see that the degree of $\Supp(\Delta_E)$ is bounded from above, thus $(E,\Supp(\Delta_E))$ is log bounded. Finally $-(K_E+\Delta_E)\sim_{\bQ} A_{X,\Delta}(E)\cdot L\ge \varepsilon L$, thus the Weil divisor $mL$ also has bounded degree, hence the triple $(E,\Supp(\Delta_E),L)$ is bounded as desired.
\end{proof}







\section{Comparison between lct and volume}

In this section we prove the following statement which gives a positive answer to \cite{HLQ-ACC}*{Question 8.1}. It will play an important role in many of the reduction steps that we will carry out in the next few sections.

\begin{thm} \label{thm:lct>=vol}
For any $n\in\bN^*$, there exists some constant $c(n)>0$ depending only on $n$ such that for any $n$-dimensional $\bQ$-Gorenstein klt singularity $x\in (X,\Delta)$ we have
\[
\lct_x(X,\Delta;\Delta)\ge c(n)\cdot \hvol(x,X,\Delta).
\]
\end{thm}

The proof will be divided into several steps, but first we shall consider an interpolation (based on a construction from \cite{XZ-uniqueness}) of the the $\bQ$-divisor $\Delta$ and the valuation ideals of the minimizing valuations of the normalized volume function. For this purpose we first revisit some results from \cite{XZ-uniqueness}.

\subsection{A multiplicity formula}

Let $x\in (X,\Delta)$ be a singularity of dimension $n$ and let $\fab$, $\fbb$ be two graded sequence of ideals. Following \cite{XZ-uniqueness}*{Section 3.3}, we define $\fa_{\bullet}\boxplus\fb_{\bullet}$ to be the graded sequence of ideals with
\[
(\fa_\bullet\boxplus\fb_\bullet)_m=\sum_{i=0}^m \fa_i\cap\fb_{m-i}.
\]
In this section, we give slight generalizations of some results from \cite{XZ-uniqueness}*{Section 3.3}.

\begin{lem} \label{lem:lct sum}
Assume that $x\in (X,\Delta)$ is klt. Then $\lct_x(\fab\boxplus \fbb)\le \lct_x(\fab)+\lct_x(\fbb)$.
\end{lem}

\begin{proof}
This follows from the same proof of \cite{XZ-uniqueness}*{Theorem 3.11}: the $\fm_x$-primary assumption there is only used to ensure that all the lct are taken at $x$.
\end{proof}

Assume next that both sequences $\fab$, $\fbb$ are decreasing and that $\fab$ is $\fm_x$-primary. Note that the latter condition implies that $\fab\boxplus\fbb$ is also $\fm_x$-primary. We will give a formula for the multiplicity of $\fab\boxplus\fbb$. Let $R_m:=\cO_{X,x}/\fa_m$. The decreasing sequence $\fbb$ induces an $\bN$-filtration $\cF$ on all $R_m$:
\[
\cF^j R_m := (\fb_j+\fa_m)/\fa_m,
\]
which also induces a filtration (still denoted as $\cF$) on the subspaces $\fa_{m-1}/\fa_m\subseteq R_m$. It is not hard to check that $\cF^j (\fa_{m-1}/\fa_m) \cong \fa_{m-1}\cap \fb_j/\fa_m\cap \fb_j$. We further assume that:
\begin{enumerate}
    \item[($\dagger$)] the limit
    \[
    \vol(\fab;\fbb^t):=\lim_{m\to \infty} \frac{\ell(\cF^{mt} R_m)}{m^n/n!} 
    \]
    exists for all $t>0$.
\end{enumerate}
Note that under this assumption, the function $t\mapsto \vol(\fab;\fbb^t)$ is decreasing and hence continuous at almost all $t$.

\begin{lem} \label{lem:mult sum}
Under the above assumptions, we have
\[
\mult(\fab\boxplus\fbb)=\mult(\fab)-(n+1) \int_0^\infty \frac{\vol(\fab;\fbb^t)}{(1+t)^{n+2}} \rd t
\]
\end{lem}

\begin{proof}
This follows from the same proof of \cite{XZ-uniqueness}*{Lemma 3.13}. We sketch the main steps for the reader's convenience. Let $R=\cO_{X,x}$ and let $\fcb=\fab\boxplus\fbb$. We have
\[
\mult(\fcb)=\lim_{m\to\infty} \frac{\ell (R/\fc_m)}{m^n/n!}=\lim_{m\to\infty} \frac{\sum_{j=1}^m \ell (R/\fc_j)}{m^{n+1}/(n+1)!}.
\]
From the short exact sequence
\[
0\to \frac{\fa_{j-\ell-1}\cap \fb_{\ell+1}}{\fa_{j-\ell}\cap \fb_{\ell+1}} \to \frac{R}{\sum_{i=0}^\ell \fa_{j-i}\cap \fb_i}\to \frac{R}{\sum_{i=0}^{\ell+1} \fa_{j-i}\cap \fb_i}\to 0
\]
for all $\ell$, we get
\[
\ell(R/\fc_j)=\ell(R/\fa_j)-\sum_{i=1}^{j} \ell( \cF^i (\fa_{j-i}/\fa_{j-i+1}) ).
\]
Summing over $j=0,1,\cdots,m$ we obtain
\[
\sum_{j=1}^m \ell (R/\fc_j) = \sum_{j=1}^m \ell (R/\fa_j) - \sum_{i=1}^m \ell( \cF^i (R/\fa_{m-i+1}) )
\]
Note that $\lim\limits_{m\to\infty} \frac{\ell( \cF^{\lceil my\rceil} (R/\fa_{m-\lceil my\rceil + 1}) )}{m^n/n!}=(1-y)^n\vol(\fab;\fbb^{\frac{y}{1-y}})$ for almost all $0<y<1$ (i.e. wherever the right hand side is continuous). Thus after dividing the above equality by $\frac{m^{n+1}}{(n+1)!}$ and using the dominated convergence theorem, we get
\begin{align*}
    \mult(\fcb) & = \mult(\fab) - (n+1)\int_0^1 (1-y)^n\vol(\fab;\fbb^{\frac{y}{1-y}}) \rd y \\
    & = \mult(\fab)-(n+1) \int_0^\infty \frac{\vol(\fab;\fbb^t)}{(1+t)^{n+2}} \rd t
\end{align*}
as desired.
\end{proof}

\subsection{Izumi inequality}

The next ingredient in the proof of Theorem \ref{thm:lct>=vol} is an Izumi type inequality as follows.

\begin{lem} \label{lem:uniform Izumi}
Let $x\in (X,\Delta)$ be an $n$-dimensional klt singularity, and let $v_0\in \Val_{X,x}^*$ be a valuation. Then there exists some constant $c_0=c_0(n)>0$ depending only on $n$ such that 
\[
\lct_x(X,\Delta;D)\ge c_0\cdot  \frac{\hvol(x,X,\Delta)}{\hvol_{X,\Delta}(v_0)}  \cdot \frac{A_{X,\Delta}(v_0)}{v_0(D)}
\]
for any effective $\bQ$-Cartier divisor $D$ on $X$.
\end{lem}

As an immediate corollary, we have

\begin{cor} \label{cor:Izumi via vol minimizer}
Let $n\in \bN^*$. Then there exists some constant $c_0>0$ depending only on $n$ such that for any $n$-dimensional klt singularity $x\in (X,\Delta)$ and any effective $\bQ$-Cartier divisor $D$ on $X$, we have 
\[
\lct_x(X,\Delta;D)\ge c_0\cdot \frac{A_{X,\Delta}(v_0)}{v_0(D)},
\]
where $v_0$ is the minimizing valuation of the normalized volume function, i.e. $\hvol_{X,\Delta}(v_0)=\hvol(x,X,\Delta)$.
\end{cor}

This can be seen as a uniform Izumi type estimate. Recall that the classical Izumi inequality (see e.g. \cite{Laz-positivity-2}*{Proposition 9.5.13}) says that if $x\in X$ is a smooth point and $D$ is an effective divisor on $X$, then $\lct_x(X;D)\ge \frac{1}{\mult_x D}$. More generally, for any klt singularity $x\in (X,\Delta)$, there exists some constant $c>0$ such that
\[
\lct_x(X,\Delta;D)\ge  \frac{c}{\mult_x D}
\]
for any effective $\bQ$-Cartier divisor $D$ on $X$ (see e.g. \cite{Li-normalized-volume}*{Theorem 3.1}). However, the constant $c$ in general depends on the singularity $x\in (X,\Delta)$. Therefore, the above corollary asserts that a uniform constant can be achieved if we replace $\mult_x$ by the minimizing valuation of the normalized volume function. Similarly, Lemma \ref{lem:uniform Izumi} suggests that the constant in the Izumi inequality tend to get worse if we choose a valuation that's further away from the normalized volume minimizer.

\begin{proof}[Proof of Lemma \ref{lem:uniform Izumi}]
After rescaling the coefficient, we may assume that $D=(f=0)$ is Cartier. For ease of notation, we will abbreviate $\lct_x(X,\Delta;\cdot)$ as $\lct(\cdot)$. Let $\fab=\fa_\bullet (v_0)$ be the graded sequence of valuation ideals of $v_0$, i.e. $\fa_m=\{s\in \cO_{X,x}\,|\,v_0(s)\ge m\}$. For each $t>0$, we also set $\fb_{m,t}=(f)^{\lceil\frac{m}{t}\rceil}$ and $\fc_{\bullet,t}=\fa_\bullet(v_0)\boxplus\fb_{\bullet,t}$. Roughly speaking, $\fc_{m,t}$ is generated by $s\in \cO_{X,x}$ such that $v_0(s) + t\cdot \ord_D(s)\ge m$, so as $t$ varies they interpolate between $\fab (v_0)$ and the ideals $\cO_X(-mD)$ ($m\in\bN$). The idea of the proof is to analyze the inequality (see \cite{Liu18}*{Theorem 27})
\[
\lct(\fc_{\bullet,t})^n\cdot \mult(\fc_{\bullet,t}) \ge \hvol(x,X,\Delta)
\]
for some suitably chosen value of $t$. Note that $\lct(\fb_{\bullet,t})=t\cdot \lct(D)$ and $\lct(\fab)\le A_{X,\Delta}(v_0)$ (the latter follows from the fact that $v_0(\fab(v_0))=1$). We take
\[
t=\frac{A_{X,\Delta}(v_0)}{\lct(D)}>0
\]
and write $\fb_{\bullet,t}$, $\fc_{\bullet,t}$ simply as $\fbb$, $\fcb$ from now on. We then have
\begin{equation} \label{eq:lct<=2A(v_0)}
    \lct(\fcb)\le \lct(\fab)+\lct(\fbb)\le 2A_{X,\Delta}(v_0)
\end{equation}
by Lemma \ref{lem:lct sum} and the above discussions. We claim that
\begin{equation} \label{eq:mult c_t}
    \frac{\mult(\fcb)}{\vol(v_0)}=1-(n+1)\int_0^{1/c}\frac{(1-cu)^n}{(1+u)^{n+2}}\rd u.
\end{equation}
where $c=\frac{v_0(D)}{t}=\lct(D)\cdot \left(\frac{A_{X,\Delta}(v_0)}{v_0(D)}\right)^{-1}$. Granting this for the moment, let us finish the proof of the lemma. To this end, denote the right hand side of \eqref{eq:mult c_t} by $S(c)$ and treat it as a function of $c>0$. Note that $S(c)\ge 1-\int_0^{1/c}\frac{n+1}{(1+u)^{n+2}}\rd u = \int_{1/c}^{+\infty}\frac{n+1}{(1+u)^{n+2}}\rd u$, hence $S(c)>0$ and $\lim\limits_{c\to \infty} S(c)=1$. We also have
\begin{equation} \label{eq:lim S(c)/c}
    \lim_{c\to 0} \frac{S(c)}{c} = 1.
\end{equation}
Indeed, we have
\begin{align*}
    S(c) & = 1 - (n+1) \int_0^{1/c} \frac{(1+c-c(1+u))^n}{(1+u)^{n+2}} \rd u \\
    & = 1 - (n+1)\sum_{i=0}^n \int_0^{1/c} \frac{\binom{n}{i}(1+c)^{n-i}(-c)^i}{(1+u)^{n+2-i}} \rd u \\
    & = 1 - (n+1) \int_0^{1/c} \left( \frac{1+cn}{(1+u)^{n+2}} - \frac{cn}{(1+u)^{n+1}} \right)\rd u + O(c^2),
\end{align*}
from which \eqref{eq:lim S(c)/c} follows easily (after a direct calculation). It then follows that there exists some constant $M>0$ depending only on $n$ such that $S(c)\le M\cdot c$ for all $c>0$.

On the other hand, by \eqref{eq:lct<=2A(v_0)} and \eqref{eq:mult c_t} we have
\[
\hvol(x,X,\Delta)\le \lct(\fcb)^n\cdot \mult(\fcb) \le 2^n S(c)\cdot \hvol_{X,\Delta}(v_0)\le 2^n M c\cdot \hvol_{X,\Delta}(v_0).
\]
Therefore, 
\[
c\ge \frac{1}{2^n M}\cdot  \frac{\hvol(x,X,\Delta)}{\hvol_{X,\Delta}(v_0)} .
\]
Recall that $c=\lct(D)\cdot \left(\frac{A_{X,\Delta}(v_0)}{v_0(D)}\right)^{-1}$, this gives the statement of the lemma with $c_0=\frac{1}{2^n M}$ (which only depends on the dimension $n$). 

It remains to prove the multiplicity formula \eqref{eq:mult c_t} using Lemma \ref{lem:mult sum}. First let us verify that the assumption $(\dagger)$ in Lemma \ref{lem:mult sum} is satisfied. Recall that (using the notation in Lemma \ref{lem:mult sum})
\[
\cF^j R_m = (\fb_j+\fa_m)/\fa_m\cong \fb_j/(\fb_j\cap \fa_m).
\]
Since $\fb_j=(f)^{\lceil\frac{j}{t}\rceil}$ and $\fa_m=\{s\in R\,|\,v_0(s)\ge m\}$, we see that $s\in \fb_j\cap \fa_m$ if and only if $s=f^{\lceil\frac{j}{t}\rceil}\cdot s_1$ for some $s_1\in R$ with $v_0(s_1)\ge m-\lceil\frac{j}{t}\rceil v_0(f)$. Thus
\[
\cF^j R_m \cong (f)^{\lceil\frac{j}{t}\rceil} / (f)^{\lceil\frac{j}{t}\rceil} \cdot \fa_{m-\lceil\frac{j}{t}\rceil v_0(f)} \cong R/\fa_{m-\lceil\frac{j}{t}\rceil v_0(f)}.
\]
It is then clear that 
\[
\lim_{m\to \infty} \frac{\ell(\cF^{mu} R_m)}{m^n/n!} = \max\{0,(1-cu)^n\} \cdot \vol(v_0),
\]
hence Lemma \ref{lem:mult sum} applies and gives
\[
\mult(\fcb) = \mult(\fab)-(n+1) \int_0^{1/c} \frac{(1-cu)^n \cdot \vol(v_0)}{(1+u)^{n+2}} \rd u.
\]
As $\mult(\fab)=\vol(v_0)$, this is exactly \eqref{eq:mult c_t}. The proof is now complete.
\end{proof}

\subsection{Completion of the proof}

We are now ready to give the proof of Theorem \ref{thm:lct>=vol}.

\begin{proof}[Proof of Theorem \ref{thm:lct>=vol}]
By \cite{LX-cubic-3fold}*{Lemma A.1}, there exists some $v_0\in \Val^*_{X,x}$ such that $\hvol_X(v_0)\le n^n$. Since $(X,\Delta)$ is klt, we have $v_0(\Delta)<A_X(v_0)$. Clearly we also have $A_{X,\Delta}(v_0)\le A_X(v_0)$. Hence by Lemma \ref{lem:uniform Izumi} we have
\begin{align*}
    \lct_x(X,\Delta;\Delta) & \ge c_0(n)\cdot  \frac{\hvol(x,X,\Delta)}{\hvol_{X,\Delta}(v_0)}  \cdot \frac{A_{X,\Delta}(v_0)}{v_0(\Delta)} \\
    & = \frac{c_0(n)}{A_{X,\Delta}(v_0)^{n-1}\cdot v_0(\Delta)\cdot \vol(v_0)} \cdot \hvol(x,X,\Delta)\\
    & \ge \frac{c_0(n)}{A_X(v_0)^{n-1}\cdot A_X(v_0)\cdot \vol(v_0)} \cdot \hvol(x,X,\Delta)\\
    & = \frac{c_0(n)}{\hvol_X(v_0)} \cdot \hvol(x,X,\Delta)\\
    & \ge \frac{c_0(n)}{n^n} \cdot \hvol(x,X,\Delta)
\end{align*}
for some $c_0(n)>0$ that only depends on $n$. Thus the theorem holds with $c(n)=\frac{c_0(n)}{n^n}$.
\end{proof}

\section{Boundedness of singularities} \label{sect:bdd via mld^K}

In this section, we study the boundedness of klt singularities using the minimal log discrepancies of their Koll\'ar components. The first result is the following theorem, which gives a criterion for special boundedness of singularities.

\begin{thm} \label{thm:bdd via vol and mldK}
Let $\varepsilon,A>0$, let $n$ be a positive integer, and let $I\subseteq [0,1]\cap \bQ$ be a finite set. Then there exists some constant $\delta=\delta(n,\varepsilon,A,I)>0$ such that any $n$-dimensional klt singularity $x\in (X,\Delta)$ with $\hvol(x,X,\Delta)\ge \varepsilon$, $\mldk(x,X,\Delta)\le A$ and $\Coef(\Delta)\subseteq I$ admits a $\delta$-plt blowup.
\end{thm}

Let us first note that having bounded $\mldk$ is a necessary condition for the above conclusion.

\begin{lem} \label{lem:mldk bdd if delta-plt blowup exist}
Let $n\in\bN^*$ and let $\delta>0$. Then there exists some $A=A(n,\delta)>0$ such that $\mldk(x,X,\Delta)\le A$ for any $n$-dimensional klt singularity $x\in (X,\Delta)$ that admits a $\delta$-plt blowup.
\end{lem}

\begin{proof}
By \cite{HLS-epsilon-plt-blowup}*{Proposition 4.3(1)}, there exists some constant $A=A(n,\delta)>0$ such that for any $\delta$-plt blowup $\pi\colon Y\to X$ of an $n$-dimensional klt singularity $x\in (X,\Delta)$ with exceptional divisor $E$, we have $A_{X,\Delta}(E)\le A$. In particular, $\mldk(x,X,\Delta)\le A$.
\end{proof}

The key to the proof of Theorem \ref{thm:bdd via vol and mldK} is to bound the Cartier index of divisors on the plt blowups that extract the Koll\'ar components. Using \cite{XZ-uniqueness}, this essentially boils down to the following volume estimate.

\begin{lem} \label{lem:vol comparison lc blowup}
Let $(X,\Delta)$ be a klt pair of dimension $n$ and let $\pi\colon Y\to X$ be an lc blowup with exceptional divisor $E$. Then
\[
\hvol(y,Y,\Delta_Y)\ge \frac{\hvol(\pi(y),X,\Delta)}{\max\{1,A_{X,\Delta}(E)^n\}}
\]
for all $y\in Y$, where $\Delta_Y=\pi_*^{-1}\Delta$.
\end{lem}

\begin{proof}
Let $a=A_{X,\Delta}(E)-1$ and let $x=\pi(y)$. If $a\le 0$, then $K_Y+\Delta_Y\le \pi^*(K_X+\Delta)$, hence by Lemma \ref{lem:vol under birational map}, we get $\hvol(y,Y,\Delta_Y)\ge \hvol(x,X,\Delta)$ for all $y\in E$. Thus we may assume that $a>0$. Let $m>0$ be a sufficiently divisible integer, and let $\fa=\pi_*\cO_Y(-mE)$. We further let $D=\frac{1}{mp}(\{f_1=0\}+\cdots+\{f_p=0\})$ for some sufficiently large integer $p$ and some general $f_1,\cdots,f_p\in \fa$. Since $(Y,\Delta_Y+E)$ is lc by assumption, and
\[
K_Y+\Delta_Y+E=\pi^*(K_X+\Delta+\fa^{\frac{a+1}{m}}),
\]
we see that $\lct(X,\Delta;D)= \lct(X,\Delta;\fa^{1/m})\ge a+1$ and $\lct(X,\Delta;aD)\ge 1+\frac{1}{a}$. 
By Lemma \ref{lem:vol inequality} we obtain
\[
\hvol(x,X,\Delta+aD)\ge \frac{\hvol(x,X,\Delta)}{A_{X,\Delta}(E)^n}.
\]
Notice that $\pi^*(K_X+\Delta+aD)=K_Y+\Delta_Y+aD_Y$ where $D_Y=\pi_*^{-1}D$, hence another application of Lemma \ref{lem:vol under birational map} implies 
\[
\hvol(y,Y,\Delta_Y)\ge \hvol(y,Y,\Delta_Y+aD_Y)\ge \hvol(x,X,\Delta+aD)\ge A_{X,\Delta}(E)^{-n}\cdot \hvol(x,X,\Delta)
\]
for all $y\in E$. This proves the lemma.
\end{proof}

\begin{proof}[Proof of Theorem \ref{thm:bdd via vol and mldK}]
Without loss of generality we may assume that $A\ge 1$. By assumption, there exists a Koll\'ar component $E$ over $x\in (X,\Delta)$ such that $A_{X,\Delta}(E)\le A$. Let $\pi\colon Y\to X$ be the corresponding plt blowup and let $\Delta_Y$ be the strict transform of $\Delta$. By Lemma \ref{lem:vol comparison lc blowup}, we have
\begin{equation} \label{eq:vol(y)>=}
    \hvol(y,Y,\Delta_Y)\ge \frac{\varepsilon}{A^n}
\end{equation}
for all $y\in \pi^{-1}(x)=E$. 
Thus by Lemma \ref{lem:index bound via volume}, we deduce that there exists a positive integer $N$ depending only on $\varepsilon,n,A$ and the coefficient set $I$ such that $N\Delta$ has integer coefficients and $ND$ is Cartier for any $\bQ$-Cartier Weil divisor $D$ on $Y$. In particular, $N^2(K_Y+\Delta_Y+E)$ is Cartier. Since $(Y,\Delta_Y+E)$ is plt, this implies that it is $\frac{1}{N^2}$-plt and we are done. 
\end{proof}

Using Lemma \ref{lem:delta-plt implies special bounded}, this immediately implies:

\begin{cor} 
Let $\varepsilon,A>0$, $n\in\bN^*$ and let $I\subseteq [0,1]\cap \bQ$ be a finite set. Then the set of $n$-dimensional klt singularities $x\in (X,\Delta)$ with $\hvol(x,X,\Delta)\ge \varepsilon$, $\mldk(x,X,\Delta)\le A$ and $\Coef(\Delta)\subseteq I$ is bounded up to special degeneration.
\end{cor}

For toric singularities we have the following uniform bound on $\mldk$. As an application, we get the boundedness of toric singularities whose volumes are bounded from below.

\begin{lem} \label{lem:toric mld^K}
Let $x\in (X,\Delta)$ be a klt toric singularity of dimension $n$. Then there exists a torus invariant Koll\'ar component $E$ such that $A_{X,\Delta}(E)\le n$. In particular,
\[
\mldk(x,X,\Delta)\le n.
\]
\end{lem}

\begin{proof}
This follows from two well-known results:
\begin{enumerate}
    \item $\mld(x,X,\Delta)\le n$ for toric singularities (see \cites{Borisov-toric-mld,Ambro-toric-mld}) and the mld can be computed by some torus invariant divisor over $x\in (X,\Delta)$;
    \item every torus invariant divisor over $x\in (X,\Delta)$ is a Koll\'ar component.
\end{enumerate}
For the reader's convenience we sketch the proof. We refer to \cite{Fulton-toric} for basics on toric varieties. We have $X=X(\sigma)$ for some rational convex polyhedral cone $\sigma\subseteq N_\bR=\bR^n$. Let $\rho_i\in N=\bZ^n$ be the primitive generator of the $1$-dimensional faces of $\sigma$. We may write $\Delta=\sum a_i \Delta_i$ where the $\Delta_i$'s are torus invariant prime divisors corresponding to $\rho_i$. The pair $(X,\Delta)$ is klt if and only if $0\le a_i<1$ for all $i$ and there exists a linear function $f\colon N_\bR\to \bR$ such that $f(\rho_i)=1-a_i$. Every torus invariant divisor $E_w$ over $x\in (X,\Delta)$ comes from a primitive vector $w\in N\cap \mathrm{int}(\sigma)$ and the corresponding log discrepancy is $A_{X,\Delta}(E_w)=f(w)$. In particular, if $w$ is the primitive generator in the direction $\rho_1+\cdots+\rho_n$ then clearly $A_{X,\Delta}(E_w)=f(w)\le f(\rho_1)+\cdots+f(\rho_n)\le n$.

Next, let $E$ be a torus invariant divisor over $x\in (X,\Delta)$ that corresponds to $w\in N\cap \mathrm{int}(\sigma)$. After adding the vertex $w$, any simplicial subdivision of $\sigma$ gives rise to a birational morphism $Y_1\to X$ such that $Y_1$ is $\bQ$-factorial and $E$ is the only exceptional divisor (since $\bC_{\ge 0}\cdot w$ is the only $1$-dimensional face that's new). Let $Y$ be the ample model of $-E$ over $X$, i.e. $Y_1\dashrightarrow Y$ is a birational contraction over $X$ and $-E$ is $\pi$-ample (where $\pi\colon Y\to X$); in other words, $Y$ is the outcome of the $(-E)$-MMP over $X$ followed by the corresponding relative ample model (we can run any MMP on toric varieties). Note that the toric pair $(Y,\Delta_Y+E)$ is plt (where $\Delta_Y=\pi^{-1}_* \Delta$) since $\Coef(\Delta_Y+E)\le 1$ and $E$ is the only prime divisor with coefficient $1$. It follows that $E$ is a Koll\'ar component. This completes the proof.
\end{proof}

\begin{prop} \label{prop:toric bdd}
Let $n\in \bN^*$ and let $\varepsilon>0$. Then there are only finitely many toric singularities $x\in X$ $($up to isomorphism$)$ that supports a klt singularity $x\in (X,\Delta)$ with $\hvol(x,X,\Delta)\ge \varepsilon$ $($for some effective $\bQ$-divisor $\Delta)$.
\end{prop}

Note that we do not require the boundary $\Delta$ to be torus invariant.

\begin{proof}
First we treat the $\Delta=0$ case, i.e. when $x\in X$ itself is $\bQ$-Gorenstein and $\hvol(x,X)\ge \varepsilon$. By the proof of Theorem \ref{thm:bdd via vol and mldK} and Lemma \ref{lem:toric mld^K}, we see that there exists some $\delta>0$ depending only on $n,\varepsilon$ such that $x\in X$ admits a $\delta$-plt toric blowup. If $E$ is the corresponding torus invariant Koll\'ar component, then by adjunction $(E,\Diff_E(0))$ is $\delta$-klt. By \cite{BB-toric-bab}, there are only finitely many such variety $E$. As in the proof of Lemma \ref{lem:delta-plt implies special bounded}, the divisor $E$ induces a degeneration of $x\in X$ to the orbifold cone over $E$. But $E$ is torus invariant, so the degeneration is trivial (the corresponding test configuration is $X\times \bA^1$ with a diagonal $\bG_m$-action), thus $X$ itself is an orbifold cone over $E$. The finiteness of $E$ then implies the finiteness of $x\in X$ as in proof of Lemma \ref{lem:delta-plt implies special bounded}.

For the general case, let $\pi\colon Y\to X$ be a small birational modification such that $-K_Y$ is $\bQ$-Cartier and $\pi$-ample. Such modification exists and is torus equivariant by the following Lemma \ref{lem:ample model of Weil div}. Let $y\in \pi^{-1}(x)$ be a torus invariant closed point. By Lemma \ref{lem:vol under birational map}, we have $\hvol(y,Y)\ge \hvol(y,Y,\Delta_Y)\ge \hvol(x,X,\Delta)\ge \varepsilon$, thus from the boundary-free case treated above we know that there are only finitely many such toric singularities $y\in Y$. Hence there exists some constant $M\gg 0$ depending only on $n,\varepsilon$ and some torus invariant divisor $E$ over $y\in Y$ such that $A_Y(E)\le M$ and $\vol_{Y,y}(\ord_E)\le M$. 

We now view $E$ as a divisor over $x\in X$. Let $\Gamma$ be the sum of all torus invariant prime divisors on $X$, and let $\varphi\colon Z\to X$ be the associated blowup that extracts the divisor $E$. Note that $K_X+\Gamma\sim 0$ and $E$ is an lc place of $(X,\Gamma)$. Since $E$ is torus invariant, we know from the same proof above that $x\in X$ is an orbifold cone over $E$. We aim to show that there exists some integer $N>0$ depending only on $n,\varepsilon$ such that $N\cdot E$ is Cartier on $Z$. Since $-E|_E$ is ample and $(-E|_E^{n-1})=\vol_{X,x}(\ord_E)\le \vol_{Y,y}(\ord_E)\le M$ (the first inequality comes from \cite{LX-cubic-3fold}*{Lemma 2.9(1)}), it would follow that there are only finitely many such toric varieties $E$, and we get the finiteness of $x\in X$.

To this end, we choose some sufficiently divisible integer $m$, let $D_Y=\frac{1}{m}D_0\sim_\bQ -K_Y$ for some general $D_0\in |-mK_Y|$, and let $G=\frac{M-1}{M}\Gamma+\frac{1}{M}\pi_* D_Y$. Then $K_X+G\sim_\bQ 0$ and by Bertini theorem we know that $(Y,\Gamma_Y+D_Y)$ is lc (as usual $\Gamma_Y$ etc. denotes the strict transform of $\Gamma$ etc. on $Y$), hence for any valuation $v\in \Val_{X,x}^*$, we get
\[
A_{Y,G_Y}(v)\ge A_{Y}(v)-\frac{M-1}{M}v(\Gamma_Y+D_Y)\ge  \frac{1}{M}A_Y(v)
\]
(the first inequality holds as long as $M\ge 2$).  It follows that $A_{X,G}(v)=A_{Y,G_Y}(v)\ge \frac{1}{M}A_Y(v)\ge \frac{1}{M}A_{X,\Delta}(v)$, thus $\hvol(x,X,G)\ge M^{-n}\hvol(x,X,\Delta)\ge M^{-n}\varepsilon$ as in Lemma \ref{lem:vol inequality}. On the other hand, recall that $E$ is an lc place of $(Y,\Gamma_Y)$, hence $A_{X,G}(E)=A_{Y,G_Y}(E)\le A_{Y,\frac{M-1}{M}\Gamma_Y}(E)=\frac{1}{M}A_Y(E)\le 1$ by our choice of $M$. This means that $K_Z+G_Z\le \varphi^*(K_X+G)$, and an application of Lemma \ref{lem:vol under birational map} gives $\hvol(z,Z,G_Z)\ge M^{-n}\varepsilon$ for all $z\in E$. By Lemma \ref{lem:index bound via volume} we deduce that the Cartier index of $E$ on $Z$ is bounded from above by some constant that only depends on $n,\varepsilon$. As explained earlier, this concludes the proof.
\end{proof}

The following result is used in the above proof, and will be needed again in the proof of Theorem \ref{thm:delta-plt blowup strong version}.

\begin{lem} \label{lem:ample model of Weil div}
Let $(X,\Delta)$ be a klt pair and let $D$ be a Weil $\bQ$-divisor on $X$. Then there exists a unique small birational modification $\pi\colon Y\to X$ such that the strict transform $D_Y=\pi^{-1}_* D$ is $\bQ$-Cartier and $\pi$-ample.
\end{lem}

\begin{proof}
The uniqueness part follows from the existence since we necessarily have 
\[
Y=\Proj_X \bigoplus_{m\in\bN} \pi_*\cO_Y(mrD_Y)=\Proj_X \bigoplus_{m\in\bN} \cO_X(mrD)
\]
for some sufficiently divisible $r\in \bN$. Every Weil divisor $D$ on $X$ can be written as $D=A-B$ where $A$ is effective and $B$ is Cartier. Since $\pi^{-1}_*(A-B) = \pi^{-1}_* A - \pi^* B$ is $\pi$-ample if and only if $\pi^{-1}_* A$ is $\pi$-ample, we may assume that $D$ is effective. Let $\pi_1\colon Y_1\to X$ be a small $\bQ$-factorial modification of $X$, which exists by \cite{BCHM}*{Corollary 1.4.3} (and the remarks thereafter). Let $\Delta_1$ (resp. $D_1$) be the strict transform of $\Delta$ (resp. $D$). Note that $(Y_1,\Delta_1)$ is klt (it is crepant to $(X,\Delta)$), thus the pair $(Y_1,\Delta_1+\varepsilon D_1)$ is also klt for $0<\varepsilon\ll 1$. By \cite{BCHM}*{Theorem 1.2(2)}, it has a log canonical model over $X$, i.e. there exists a birational contraction $Y_1\dashrightarrow Y$ over $X$ such that $K_Y+\Delta_Y+\varepsilon D_Y$ is ample over $X$ (where $\Delta_Y, D_Y$ are the strict transforms of $\Delta,D$). Denote the map $Y\to X$ by $\pi$. Clearly it is a small birational contraction (as $\pi_1$ is), and hence $K_Y+\Delta_Y=\pi^*(K_X+\Delta)$ and 
\[
K_Y+\Delta_Y+\varepsilon D_Y \sim_{\pi,\bQ} \varepsilon D_Y.
\]
It follows that $D_Y$ is also $\pi$-ample and we finish the proof.
\end{proof}

We next show that Conjecture \ref{conj:bdd for mldK, vol>epsilon} and Theorem \ref{thm:bdd via vol and mldK} together imply Conjecture \ref{conj:vol>epsilon implies bounded}. The main ingredients come from Theorem \ref{thm:lct>=vol} and Koll\'ar's effective basepoint-free theorem \cite{Kollar-effective-bpf}.

\begin{thm} \label{thm:delta-plt blowup strong version}
Assume that Conjecture \ref{conj:bdd for mldK, vol>epsilon} holds in dimension $n$. Then for any $\varepsilon>0$ and any $n$-dimensional klt singularity $x\in (X,\Delta=\sum_{i=1}^m a_i \Delta_i)$ such that
\begin{enumerate}
    \item $a_i\ge \varepsilon$ for all $i$,
    \item each $\Delta_i$ is an effective Weil divisor, and
    \item $\hvol(x,X,\Delta)\ge \varepsilon$,
\end{enumerate}
there exists some $\delta=\delta(n,\varepsilon)>0$ such that $x\in (X,\Delta)$ admits a $\delta$-plt blowup. In particular, the set of such singularities is log bounded up to special degeneration.
\end{thm}

\begin{proof}
We focus on the existence of $\delta$-plt blowup, since the special boundedness would then follow from Lemma \ref{lem:delta-plt implies special bounded}. If $\Delta^+\ge \Delta$ and $(X,\Delta^+)$ is klt, then any $\delta$-plt blowup of $(X,\Delta^+)$ is also a $\delta$-plt blowup of $(X,\Delta)$. By Conjecture \ref{conj:bdd for mldK, vol>epsilon} and Theorem \ref{thm:bdd via vol and mldK}, if $r\Delta^+$ has integral coefficients for some integer $r>0$ and $\hvol(x,X,\Delta^+)\ge \varepsilon_0$ for some $\varepsilon_0>0$, then we may choose $\delta$ to only depend on $n,r$ and $\varepsilon_0$. So our goal is to prove that we can find $\Delta^+$, $r$ and $\varepsilon_0$ as above such that $r,\varepsilon_0$ only depends on $n,\varepsilon$.

To this end, let $\gamma=c(n)\cdot \varepsilon$ where $c(n)>0$ is the constant from Theorem \ref{thm:lct>=vol}. Let $\ell=\lceil \frac{2+\gamma}{\varepsilon \gamma} \rceil$, and let 
\[
\Delta'=\sum_{i=1}^m \frac{1}{\ell}\left\lfloor\frac{(1+\gamma)\ell a_i}{1+\frac{\gamma}{2}}\right\rfloor \Delta_i.
\]
Using the assumption that $a_i\ge \varepsilon$, it is not hard to check that $\Delta\le \Delta'$ and 
\begin{equation} \label{eq:Delta'<=(1+gamma)Delta}
    \left(1+\frac{\gamma}{2}\right)\Delta'\le (1+\gamma)\Delta.
\end{equation}
If both $K_X$ and $\Delta'$ were $\bQ$-Cartier, we could simply take $r=\ell$, $\Delta^+=\Delta'$: by Theorem \ref{thm:lct>=vol}, the pair $(X,(1+\gamma)\Delta)$ is klt, hence so is $(X,(1+\frac{\gamma}{2})\Delta')$ by \eqref{eq:Delta'<=(1+gamma)Delta}. By Lemma \ref{lem:vol inequality} (with $\lambda=\frac{\gamma}{2}$), we get $\hvol(x,X,\Delta')\ge \left( \frac{\gamma}{\gamma+2} \right)^n \hvol(x,X,\Delta)\ge \varepsilon_0$ for some constant $\varepsilon_0$ that depends only on $n,\varepsilon$. In other words, the desired conditions on $r,\Delta^+:=\Delta'$ and $\varepsilon$ are satisfied.

To deal with the general case, the idea is to find another divisor $D\ge 0$ such that $(X,\Delta'+D)$ is klt (in particular, $K_X+\Delta'+D$ is $\bQ$-Cartier), while keeping the coefficients in a fixed finite set. To this end, we take a $\bQ$-factorial modification $\pi\colon Y\to X$ such that $-(K_Y+\Delta'_Y)$ is $\pi$-nef where $\Delta'_Y=\pi^{-1}_* \Delta'$. Such map exists by Lemma \ref{lem:ample model of Weil div}: first take a small birational modification such that the strict transform of $-(K_X+\Delta')$ is ample over $X$, then take a further small $\bQ$-factorial modification using \cite{BCHM}*{Corollary 1.4.3}. We also set $\Delta_Y=\pi^{-1}_* \Delta$. The strict transform of the sought divisor $D$ should be in the $\bQ$-linear system $|-(K_Y+\Delta'_Y)|_\bQ$. To control the coefficient of $D$ we first proceed to verify the effective basepoint-freeness of this linear system. Since $(Y,\Delta_Y)$ is crepant to $(X,\Delta)$, we have 
\[
\hvol(y,Y,\Delta_Y)\ge \hvol(\pi(y),X,\Delta)\ge \varepsilon
\]
for all $y\in Y$ by Lemma \ref{lem:vol under birational map}. Since $\ell \Delta'_Y$ has integer coefficients, we deduce from Lemma \ref{lem:index bound via volume} that there exists a positive integer $N_0$ depending only on $n$ and $\varepsilon$ such that $L:=-N_0(K_Y+\Delta'_Y)$ is Cartier. Note that $L$ and $L-(K_Y+\Delta'_Y)$ are both $\pi$-nef and $\pi$-big by our construction of $Y$, therefore by Koll\'ar's effective base-point-free theorem \cite{Kollar-effective-bpf}*{Theorem 1.1} (see also \cite{Fujino-effective-bpf}*{Theorem 1.3} for the relative version), there exists another positive integer $m_0$ depending only on the dimension $n$ such that $m_0 L$ is $\pi$-generated. In particular, we get an integer $r_0=m_0 N_0$ which only depends on $n$ and $\varepsilon$ such that $-r_0(K_Y+\Delta'_Y)$ is Cartier and $\pi$-generated. By the same argument as in the special case above (where $K_X$ and $\Delta'$ are assumed to be $\bQ$-Cartier), we also know that $(Y,(1+\frac{\gamma}{2})\Delta'_Y)$ is klt. Thus by Bertini's theorem, we can choose some effective $\bQ$-divisor $D_Y\sim_{\pi,\bQ} -(K_Y+\Delta'_Y)$ such that $2r_0 D_Y$ has integral coefficients and the pair $(Y,(1+\frac{\gamma_0}{2})(\Delta'_Y+D_Y))$ is klt for $\gamma_0=\min\{1,\gamma\}$. 

As $K_Y+\Delta'_Y+D_Y\sim_{\pi,\bQ} 0$, we have
\[
K_Y+\Delta'_Y+D_Y = \pi^*(K_X+\Delta'+D)
\]
where $D=\pi_* D_Y$. Note that $(X,\Delta'+D)$ is klt since the same holds for $(Y,\Delta'_Y+D_Y)$. We also know that $2r_0\ell (\Delta'+D)$ has integral coefficients by our construction. According to the discussion at the beginning of the proof, it remains to check that $\hvol(x,X,\Delta'+D)\ge \varepsilon_0$ for some constant $\varepsilon_0>0$ that only depends on $n$ and $\varepsilon$. But by construction, it is not hard to see that $\Delta'+D-\Delta$ is $\bQ$-Cartier and that $(X,\Delta+(1+\frac{\gamma_0}{2})(\Delta'+D-\Delta))$ is klt: indeed, the strict transform of the boundary on $Y$ is at most $(1+\frac{\gamma_0}{2})(\Delta'_Y+D_Y)$. Hence by Lemma \ref{lem:vol inequality} (with $\lambda=\frac{\gamma_0}{2}$) we obtain
\[
\hvol(x,X,\Delta'+D)\ge \left(\frac{\gamma_0}{2+\gamma_0} \right)^n \hvol(x,X,\Delta)\ge \varepsilon_0
\]
for some constant that only depends on $n$ and $\varepsilon$ (as the same holds for $\gamma_0$). This finishes the proof.
\end{proof}

\begin{rem}
While the results in this paper are stated for pairs with rational coefficients, the general real coefficient case can often be reduced to the rational case by perturbing the coefficients. In Conjecture \ref{conj:bdd for mldK, vol>epsilon} we can even relax the assumption that $\Coef(\Delta)$ belongs to a fixed finite set $I$ to $\Coef(\Delta)\ge \varepsilon$ and get an equivalent conjecture. The reason is as follows: by perturbation, we may further assume $\Coef(\Delta)\subseteq \bQ$; if the original version of Conjecture \ref{conj:bdd for mldK, vol>epsilon} holds, then by Theorem \ref{thm:delta-plt blowup strong version} and Lemma \ref{lem:mldk bdd if delta-plt blowup exist}, for any klt pair $(X,\Delta)$ with $\Coef(\Delta)\ge \varepsilon$ and any $\eta\in X$ with $\hvol(\eta,X,\Delta)\ge \varepsilon$, we have $\mld^K(\eta,X,\Delta)\le A$ for some constant $A=A(n,\varepsilon)>0$.

Using \cite{HLS-epsilon-plt-blowup}, we can also extend Theorem \ref{thm:bdd via vol and mldK} to the real coefficient case. Indeed, given any finite set $I\subseteq [0,1]$ (not necessarily $\subseteq \bQ$), by \cite{HLS-epsilon-plt-blowup}*{Theorem 5.6} we can find a finite set $I'\subseteq [0,1]\cap\bQ$ such that: for any $n$-dimensional klt singularity $x\in (X,\Delta)$ with $\Coef(\Delta)\subseteq I$, there exists some effective $\bQ$-divisor $\Delta'\ge \Delta$ on $X$ with $\Coef(\Delta')\subseteq I'$, such that $x\in (X,\Delta')$ is klt, $\hvol(x,X,\Delta')\ge 2^{-n}\hvol(x,X,\Delta)$ and $\mldk(x,X,\Delta')\le 2\cdot \mldk(x,X,\Delta)$ (c.f. \cite{Z-mld^K-2}*{Lemma 2.17}). Since $\Delta'\ge \Delta$, any $\delta$-plt blowup of $(X,\Delta')$ is also a $\delta$-plt blowup of $(X,\Delta)$. The real coefficient case of Theorem \ref{thm:bdd via vol and mldK} then follows from the rational case applied to the coefficient set $I'$.
\end{rem}

\section{Reduction steps}

In this section, we work out some reduction steps for Conjectures \ref{conj:bdd for mldK, vol>epsilon} and \ref{conj:bdd for mldK} which apply in any dimension. They will be combined with classification results in the next section to prove both conjectures in low dimensions.

\subsection{Special complements}

One of our main tools is the notion of special complements. It helps us descend Koll\'ar components over birational models of the singularity to Koll\'ar components of the singularity itself. This notion was introduced in \cite{LXZ-HRFG} to prove the Higher Rank Finite Generation Conjecture, here we need a slight variant.

\begin{defn}
Let $(X,\Delta)$ be a klt pair and let $\pi\colon Y\to X$ be a proper birational morphism. A $\bQ$-complement $\Gamma$ of $(X,\Delta)$ is said to be \emph{special} (with respect to $\pi$) if for any $y\in \Ex(\pi)$, there exists some effective $\bQ$-Cartier $\pi$-ample $\bQ$-divisor $G\le \pi^* \Gamma$ such that $y\not\in \Supp(G)$.
\end{defn}

\begin{lem} \label{lem:kc descend special complement}
Let $(X,\Delta)$ be a klt pair, let $\pi\colon Y\to X$ be a proper birational morphism and let $(Y,\Delta_Y)$ be the crepant pullback of $(X,\Delta)$. Let $\Gamma$ be a special $\bQ$-complement with respect to $\pi$, and let $E$ be an lc place of $(X,\Delta+\Gamma)$. Assume that $E$ is of plt type over $(Y,\Delta_Y)$. Then $E$ is also of plt type over $(X,\Delta)$. 
\end{lem}

\begin{proof}
This basically follows from the same proof of \cite{LXZ-HRFG}*{Lemma 3.5}, which we provide here for reader's convenience. By Lemma \ref{lem:kc exist on every dlt model} it suffices to find an effective divisor $D$ such that $E$ is the unique lc place of $(X,\Delta+\lambda D)$, where $\lambda=\lct(X,\Delta;D)$. 

Let $\rho\colon Z\to Y$ be the plt blowup of $E$. By assumption, there exists an effective $\pi$-ample $\bQ$-divisor $G$ on $Y$ such that $\pi^*\Gamma\ge G$ and $\rho(E)\not\subseteq \Supp(G)$. Let $H$ be a sufficiently ample divisor on $X$ such that $G+\pi^*H$ is ample. Since $-E$ is $\rho$-ample, there exists some rational number $\varepsilon>0$ such that $\rho^*(G+\pi^*H)-\varepsilon E$ is ample on $Z$. Let $G_1$ be a general divisor in the $\bQ$-linear system $|\rho^*(G+\pi^*H)-\varepsilon E|_\bQ$ and consider the effective divisor $D$ on $X$ satisfying $\rho^*\pi^*D=\rho^*(\pi^*\Gamma-G)+G_1+\varepsilon E$ (this is possible since the right hand side is $\sim_\bQ 0$ over $X$). We claim that this divisor $D$ satisfies the desired condition.

Let $K_Z+\Delta_Z=\rho^*(K_Y+\Delta_Y)$ be the crepant pullback. We first note that the above claim is a consequence of the following two properties:
\begin{enumerate}
    \item $(Y,\Delta_Y+\pi^*\Gamma-G)$ is sub-lc and $E$ is an lc place of this sub-pair;
    \item $E$ is the only divisor that computes $\lct(Y,\Delta_Y;\rho_*(G_1+\varepsilon E))$.
\end{enumerate}
This is because, (1) implies that 
\[
A_{X,\Delta}(F)=A_{Y,\Delta_Y}(F)\ge \ord_F(\pi^*\Gamma-G)
\]
for divisor $F$ over $X$, and the equality holds when $F=E$; on the other hand, if we let $\mu=\lct(Y,\Delta_Y;\rho_*(G_1+\varepsilon E))>0$, then (2) implies that 
\[
A_{X,\Delta}(F)=A_{Y,\Delta_Y}(F)\ge \mu\cdot \ord_F(\rho_*(G_1+\varepsilon E)),
\]
and equality holds if and only if $F=E$. Combining the two inequalities we have
\[
\ord_F(\pi^*D)=\ord_F(\pi^*\Gamma-G+\rho_*(G_1+\varepsilon E))\le (1+\mu^{-1})A_{X,\Delta}(F),
\]
and equality holds if and only if $F=E$. In particular, $\lct(X,\Delta;D)=\frac{1}{1+\mu^{-1}}$ and $E$ is the unique lc place that computes this lct, which is exactly what we want.

It remains to prove the two properties above. Point (1) is quite straightforward since by assumption $E$ is an lc place of the sub-lc sub-pair $(Y,\Delta_Y+\pi^*\Gamma)$ and $G$ does not contain the center of $E$. To see point (2), we note that by assumption the sub-pair $(Z,\Delta_Z\vee E)$ is plt. Here we denote by $D_1\vee D_2$ the smallest $\bQ$-divisor $D$ such that $D\ge D_i$ for $i=1,2$. Let $t=\frac{A_{Y,\Delta_Y}(E)}{\varepsilon}$. Then
\[
\rho^*(K_Y+\Delta_Y+t\rho_*(G_1+\varepsilon E))=K_Z+\Delta_Z\vee E+t G_1
\]
by construction. Since $G_1$ is general, the pair $(Z,\Delta_Z\vee E+t G_1)$ is also plt. This implies that $\lct(Y,\Delta_Y;\rho_*(G_1+\varepsilon E))=t$ and $E$ is the only divisor that computes the lct. In particular (2) holds. The proof is now complete.
\end{proof}

As an application, we have:

\begin{lem} \label{lem:kc descend along lc blowup}
Let $(X,\Delta)$ be a klt pair, let $\pi\colon Y\to X$ be an lc blowup with exceptional divisor $E$, and let $\Delta_Y=\pi^{-1}_*\Delta$. Then any lc place of $(Y,\Delta_Y+E)$ that is of plt type over $(Y,\Delta_Y)$ is also of plt type over $(X,\Delta)$.
\end{lem}

\begin{proof}
By assumption $(Y,\Delta_Y+E)$ is lc and its lc centers are contained in $E$, thus $(Y,\Delta_Y)$ is klt. Let $c=A_{X,\Delta}(E)>0$ and let $\Gamma=\Delta_Y+(1-c)E<\Delta_Y+E$. Let $F$ be an lc place of $(Y,\Delta_Y+E)$ that is of plt type over $(Y,\Delta_Y)$. Since $\Gamma$ is a convex combination of $\Delta_Y$ and $\Delta_Y+E$, by interpolation we know that $F$ is also of plt type over $(Y,\Gamma)$. Note that $K_Y+\Gamma=\pi^*(K_X+\Delta)$. By Lemma \ref{lem:kc descend special complement}, it suffices to show that $F$ is an lc place of some special complement with respect to $\pi$. Since $-E$ is $\pi$-ample, by Bertini theorem we can choose some general effective $\bQ$-divisor $D_i\sim_{\pi,\bQ} -E$ $(i=1,\cdots,n=\dim X)$ on $Y$ such that $(Y,\Delta_Y+E+cD_i)$ remains lc for all $i$ and that for any $y\in \Ex(\pi)$ we have $y\not\in\Supp(D_i)$ for some $i\in\{1,\cdots,n\}$. Let $D_Y=\frac{1}{n}(D_1+\cdots+D_n)\sim_{\pi,\bQ} -E$ and let $D=\pi_* D_Y$. Then $\pi^*D=D_Y+E$ and hence $\pi^*(K_X+\Delta+cD)=K_Y+\Delta_Y+cD_Y+E$. By construction, locally on $X$ the divisor $cD$ is a special $\bQ$-complement that has $F$ as an lc place, hence we are done.
\end{proof}

In the rest of this subsection, we include two auxiliary results that will be useful later.

\begin{lem} \label{lem:closure of plt blowup over min lc center}
Let $(X,\Delta)$ be a klt pair, let $D$ be a $\bQ$-complement of $(X,\Delta)$ and let $W\subseteq X$ be a minimal lc center of $(X,\Delta+D)$. Let $E$ be an lc place of $(X,\Delta+D)$ with center $W$. Assume that $E$ is of plt type over $(X,\Delta)$ in a neighbourhood of the generic point of $W$. Then $E$ is of plt type over $(X,\Delta)$.
\end{lem}

\begin{proof}
Let $\pi\colon Y\to X$ be the prime blowup of $E$ (which exists by Lemma \ref{lem:BCHM extract divisor}). Since $E$ is an lc place of the lc pair $(X,\Delta+D)$, we have
\[
K_Y+\Delta_Y+D_Y+E=\pi^*(K_X+\Delta+D)
\]
and $(Y,\Delta_Y+D_Y+E)$ is lc. In particular, the pair $(Y,\Delta_Y+E)$ is also lc. By assumption, it is also plt over the generic point of $W$. If it is not plt everywhere, then it has some lc center that does not dominate $W$. However, any such lc center of $(Y,\Delta_Y+E)$ is also an lc center of $(Y,\Delta_Y+D_Y+E)$ and therefore maps to an lc center of $(X,\Delta+D)$ that is strictly contained in $W$. This contradicts the assumption that $W$ is a minimal lc center. Hence $(Y,\Delta_Y+E)$ is plt.
\end{proof}

\begin{lem} \label{lem:every kc is lc place}
Let $n\in\bN^*$ and let $I\subseteq [0,1]\cap \bQ$ be a finite set. Let $(X,\Delta)$ be an lc pair of dimension $n$, let $D$ be a $\bQ$-complement of $(X,\Delta)$ and let $W\subseteq X$ be an lc center of $(X,\Delta+D)$. Assume that $\Coef(\Delta),\Coef(D)\subseteq I$. Then there exists some rational number $\varepsilon_0\in (0,1)$ depending only on $n$ and $I$ such that every lc type divisor over $(X,\Delta+(1-\varepsilon_0)D)$ with center $W$ is also an lc place of $(X,\Delta+D)$.
\end{lem}

\begin{proof}\footnote{The author learned this argument from Yuchen Liu.}
We may assume that $1\in I$. By the ACC of log canonical thresholds \cite{HMX-ACC}, there exists some $\varepsilon_0\in (0,1)$ such that if $(X_0,\Delta_0)$ is an $n$-dimensional lc pair and $D_0$ is an effective $\bQ$-Cartier divisor such that $\Coef(\Delta_0),\Coef(D_0)\subseteq I$ and $(X_0,\Delta_0+(1-\varepsilon_0)D_0)$ is lc, then $(X_0,\Delta_0+D_0)$ is also lc. Let us show that this $\varepsilon_0$ satisfies the statement of the lemma. Let $E$ be an lc type divisor over $(X,\Delta+(1-\varepsilon_0)D)$ with center $W$ and let $\pi\colon Y\to X$ be the associated prime blowup. Then $(Y,\Delta_Y+(1-\varepsilon_0)D_Y+E)$ is lc. By our choice of $\varepsilon_0$, this implies that $(Y,\Delta_Y+D_Y+E)$ is lc. Suppose that $E$ is not an lc place of $(X,\Delta+D)$. Then we have
\[
K_Y+\Delta_Y+D_Y+\lambda E = \pi^*(K_X+\Delta+D)
\]
for some $\lambda<1$. In particular, $(Y,\Delta_Y+D_Y+\lambda E)$ is lc. Moreover, $E$ contains some lc center of $(Y,\Delta_Y+D_Y+\lambda E)$ since $\pi(E)=W$ is an lc center of $(X,\Delta+D)$. As $\lambda<1$, it follows that $(Y,\Delta_Y+D_Y+E)$ cannot be lc, a contradiction. Thus $E$ is an lc place of $(X,\Delta+D)$ as desired.
\end{proof}

\subsection{Plt type lc place of bounded complements}

The goal of this subsection is to reduce Conjecture \ref{conj:bdd for mldK, vol>epsilon} to the following special case, whose statement is motivated by Lemma \ref{lem:strict N-complement}. 

\begin{conj} \label{conj:bdd mldk, Q-Gor, N-comp}
Let $n,N\in\bN^*$ and let $\varepsilon>0$. Then there exists some $A=A(n,N,\varepsilon)$ such that for any $N$-complement $D$ of an $n$-dimensional klt variety $X$ and any lc center $W$ of $(X,D)$ such that $\hvol(\eta,X)\ge \varepsilon$ $($where $\eta$ is the generic point of $W)$, there exists a Koll\'ar component $E$ over $\eta\in X$ such that
\[
A_{X,D}(E)=0, \quad\text{and}\quad A_X(E)\le A.
\]
\end{conj}

Roughly speaking, we expect that all bounded complements have plt type lc places of bounded log discrepancy. Its connection with Conjecture \ref{conj:bdd for mldK, vol>epsilon} is given by the following result.

\begin{prop} \label{prop:reduce to N-complement}
For any fixed dimension $n$, Conjecture \ref{conj:bdd for mldK, vol>epsilon} and Conjecture \ref{conj:bdd mldk, Q-Gor, N-comp} are equivalent.
\end{prop}

As the first step towards the proof, we show that Conjecture \ref{conj:bdd for mldK, vol>epsilon} can be reduced to the $\bQ$-Gorenstein case.

\begin{prop} \label{prop:reduce to Q-Gorenstein}
Fix the dimension $n$. Assume that Conjecture \ref{conj:bdd for mldK, vol>epsilon} holds when $X$ is $\bQ$-Gorenstein. Then it holds in general.
\end{prop}

\begin{proof}
Starting with any klt pair $(X,\Delta)$ in dimension $n$ whose coefficients lie in a finite set $I$ and any $\eta\in X$ with $\hvol(\eta,X,\Delta)\ge \varepsilon$, our plan is to find a small birational modification $\pi\colon Y\to X$ such that $Y$ is $\bQ$-Gorenstein, together with an $N$-complement $D$ of $\eta\in (X,\Delta)$ that is special with respect to $\pi$, where the integer $N$ only depends on $n,I$ and $\varepsilon$. Let us first explain why this is enough.

After shrinking $X$, we may assume that $\overline{\eta}$ (the closure of $\eta$) is the minimal lc center of $(X,\Delta+D)$. Let $\Delta_Y$ (resp. $D_Y$) be the strict transform of $\Delta$ (resp. $D$), let $W\subseteq \pi^{-1}(\overline{\eta})$ be a minimal lc center of $(Y,\Delta_Y+D_Y)$ and let $\xi$ be the generic point of $W$. Note that $\Coef(\Delta_Y)=\Coef(\Delta)\subseteq I$ and $\Coef(D_Y)\subseteq \frac{1}{N}\bZ \cap [0,1]$. By Lemma \ref{lem:every kc is lc place}, there exists some rational number $\varepsilon_0\in (0,1)$ depending only on $n,I$ and $N$ such that if $E$ is a Koll\'ar component over $\xi\in (Y,\Delta_Y+(1-\varepsilon_0)D_Y)$, then it is also an lc place of $(Y,\Delta_Y+D_Y)$. By Lemma \ref{lem:closure of plt blowup over min lc center}, we further deduce that $E$ is of plt type over $(Y,\Delta_Y+(1-\varepsilon_0)D_Y)$ and hence also over $(Y,\Delta_Y)$. By Lemma \ref{lem:kc descend special complement}, $E$ is a Koll\'ar component over $\eta\in (X,\Delta)$. Note that as $N$ only depends on $n,I,\varepsilon$ so does $\varepsilon_0$. On the other hand, by Lemma \ref{lem:vol under birational map} we have $\hvol(\xi,Y,\Delta_Y)\ge \hvol(\eta,X,\Delta)\ge \varepsilon$. By Lemma \ref{lem:vol inequality} we further obtain
\[
\hvol(\xi,Y,\Delta_Y+(1-\varepsilon_0)D_Y)\ge \left(\frac{\varepsilon_0}{1+\varepsilon_0} \right)^n \hvol(\xi,Y,\Delta_Y)\ge C_1
\]
for some constant $C_1>0$ that only depends on $n,I,\varepsilon$. Since we assume Conjecture \ref{conj:bdd for mldK, vol>epsilon} to hold for $\bQ$-Gorenstein singularities, the Koll\'ar component $E$ above can be chosen so that 
\[
A_{Y,\Delta_Y+(1-\varepsilon_0)D_Y}(E)\le A_0
\]
for some constant $A_0$ that only depends on $n,I,\varepsilon$ (a priori it also depends on $\varepsilon_0$ and $C_1$ but these two constants only depend on $n,I$ and $\varepsilon$). As $E$ is automatically a Koll\'ar component over $\eta\in (X,\Delta)$ and an lc place of $(Y,\Delta_Y+D_Y)$ from the above discussion, we get 
\[
A_{Y,\Delta_Y+(1-\varepsilon_0)D_Y}(E) = \varepsilon_0 \cdot A_{Y,\Delta_Y}(E) + (1-\varepsilon_0) A_{Y,\Delta_Y+D_Y}(E) = \varepsilon_0 \cdot A_{Y,\Delta_Y}(E),
\]
hence $\mldk(\eta,X,\Delta)\le A_{X,\Delta}(E)=A_{Y,\Delta_Y}(E)\le \frac{A_0}{\varepsilon_0}$. Since the right hand side only depends on $n,I,\varepsilon$, this proves Conjecture \ref{conj:bdd for mldK, vol>epsilon} in the general (non-$\bQ$-Gorenstein) case.

We now return to construct the map $\pi\colon Y\to X$ and the $N$-complement $D$. The argument is very similar to that for Theorem \ref{thm:delta-plt blowup strong version}. Let $\pi\colon Y\to X$ be a small birational modification such that $K_Y$ is $\bQ$-Cartier and $\pi$-ample (existence is guaranteed by Lemma \ref{lem:ample model of Weil div}). In particular, $Y$ is $\bQ$-Gorenstein. By Lemma \ref{lem:vol under birational map} we have $\hvol(y,Y,\Delta_Y)\ge \varepsilon$ for all $y\in Y$. Thus by Theorem \ref{thm:lct>=vol}, there exists some rational number $\gamma>0$ depending only on $n,\varepsilon$ such that $(Y,(1+\gamma)\Delta_Y)$ is klt. In addition:
\begin{itemize}
    \item By Lemma \ref{lem:index bound via volume}, there exists a positive integer $N_0$ depending only on $n$ and $\varepsilon$ such that $L:=N_0\cdot K_Y$ is Cartier.
    \item By Koll\'ar's effective base-point-free theorem, there exists another positive integer $m_0$ depending only on the dimension $n$ such that $m_0 L$ is $\pi$-generated.
\end{itemize}
Putting these two facts together with Bertini's theorem, we deduce that there exists an integer $r_0=nm_0N_0$ that only depends on $n,\varepsilon$, and an effective divisor $D_Y\sim_{\pi,\bQ} K_Y$ such that 
\begin{itemize}
    \item $r_0 D_Y$ has integer coefficients, 
    \item $(Y,(1+\gamma)\Delta_Y+\gamma D_Y)$ is klt, and
    \item for any $y\in \Ex(\pi)$ there exists some irreducible component of $D_Y$ that's $\bQ$-Cartier and ample over $X$ such that $y$ is not contained in its support.
\end{itemize}
For example, one can take $D_Y=\frac{1}{r_0}(D_1+\cdots+D_n)$ where $D_1,\cdots,D_n$ are general members of $|m_0N_0 K_Y|$. Note that by construction,
\[
K_Y+(1+\gamma)\Delta_Y+\gamma D_Y\sim_{\pi,\bQ} (1+\gamma)(K_Y+\Delta_Y)\sim_{\pi,\bQ} 0,
\]
thus $K_Y+(1+\gamma)\Delta_Y+\gamma D_Y = \pi^*(K_X+(1+\gamma)\Delta+\gamma D)$ where $D=\pi_* D_Y$, and the pair $(X,(1+\gamma)\Delta+\gamma D)$ is also klt. Since the coefficients of $(1+\gamma)\Delta+\gamma D$ is contained in a finite set that only depends on $n,I$ and $\varepsilon$, by Lemma \ref{lem:strict N-complement} we see that $\eta\in (X,(1+\gamma)\Delta+\gamma D)$ admits an $N$-complement $G$ for some integer $N$ only depends on $n,I,\varepsilon$. The divisor $\Gamma=\gamma(\Delta+D)+G$ is then an $N$-complement of $\eta\in (X,\Delta)$. It remains to show that $\Gamma$ is special with respect to $\pi$. But as $\Gamma\ge \gamma D$, this is clear from our choice of $D_Y$.
\end{proof}

We now return to the proof of Proposition \ref{prop:reduce to N-complement}.

\begin{proof}[Proof of Proposition \ref{prop:reduce to N-complement}]
By Proposition \ref{prop:reduce to Q-Gorenstein}, it suffices to show that Conjecture \ref{conj:bdd mldk, Q-Gor, N-comp} is equivalent to the $\bQ$-Gorenstein case of Conjecture \ref{conj:bdd for mldK, vol>epsilon}. First assume that Conjecture \ref{conj:bdd mldk, Q-Gor, N-comp} holds. Let $(X,\Delta)$ be an $n$-dimensional klt pair where $X$ is $\bQ$-Gorenstein and let $\eta\in X$ be such that $\hvol(\eta,X,\Delta)\ge \varepsilon$. By Theorem \ref{thm:lct>=vol}, there exists some constant $c=c(n,\varepsilon)>0$ such that $(X,(1+c)\Delta)$ is still klt. We may assume that $c\in\bQ$. By Lemma \ref{lem:strict N-complement}, there exists an integer $N$ depending only on $n$ and $\Coef(\Delta)$ such that $\eta\in (X,(1+c)\Delta)$ has an $N$-complement $D$ (a priori $N$ also depends on $c$ but we already know that $c$ only depends on $n$ and $\varepsilon$). Since Conjecture \ref{conj:bdd mldk, Q-Gor, N-comp} holds, we see that there exists some $A=A(n,N,\varepsilon)$ and some Koll\'ar component $E$ over $\eta\in X$ such that $A_{X,(1+c)\Delta+D}(E)=0$ and $A_X(E)\le A$. In particular, the constant $A$ only depends on $n,\varepsilon$ and $\Coef(\Delta)$, and clearly $A_{X,\Delta}(E)\le A_X(E)\le A$. Thus to prove Conjecture \ref{conj:bdd for mldK, vol>epsilon} for $\bQ$-Gorenstein singularities it suffices to show that $E$ is also a Koll\'ar component over $x\in (X,\Delta)$. But this is straightforward: if $\pi\colon Y\to X$ is the plt blowup that extracts $E$, then $(Y,E)$ is plt and $(Y,E+(1+c)\Delta_Y+D_Y)$ is lc (because $E$ is an lc place of $(X,(1+c)\Delta+D)$). By interpolation, it follows that $(Y,E+\Delta_Y)$ is also plt and thus $E$ is also a Koll\'ar component over $\eta\in (X,\Delta)$. This proves one direction of the equivalence.

Next assume that Conjecture \ref{conj:bdd for mldK, vol>epsilon} holds. Let $X,D,W,\eta$ be as in Conjecture \ref{conj:bdd mldk, Q-Gor, N-comp}. By Lemma \ref{lem:every kc is lc place}, there exists some rational $\varepsilon_0\in(0,1)$ depending only on $N$ such that every Koll\'ar component over $\eta\in (X,(1-\varepsilon_0)D)$ is also an lc place of $(X,D)$. Since $\hvol(\eta,X,(1-\varepsilon_0)D)\ge \left(\frac{\varepsilon_0}{1+\varepsilon_0}\right)^n \hvol(\eta,X)\ge \frac{\varepsilon_0^n\varepsilon}{(1+\varepsilon_0)^n}$ by Lemma \ref{lem:vol inequality}, and Conjecture \ref{conj:bdd for mldK, vol>epsilon} holds by assumption, we see that there exists some constant $A>0$ depending only on $n,N$ and $\varepsilon$ and some Koll\'ar component $E$ over $\eta\in (X,(1-\varepsilon_0)D)$ such that $A_{X,(1-\varepsilon_0)D}(E)\le A$ (a priori $A$ also depends on $\varepsilon_0$ and $\Coef(D)$, but these only depend on $n,N,\varepsilon$ by the above construction). From the above discussion, we have $A_{X,D}(E)=0$. Thus $A_X(E)=\frac{A_{X,(1-\varepsilon_0)D}(E)}{\varepsilon_0}\le \frac{A}{\varepsilon_0}$ is bounded from above by some constant that only depends on $n,N$ and $\varepsilon$. This proves that the other direction of the desired equivalence.
\end{proof}

\subsection{Further reduction} \label{sect:reduce to mld^lc and 1-comp}

We further break Conjecture \ref{conj:bdd mldk, Q-Gor, N-comp} into two smaller parts. The first part is the following weaker version of Conjecture \ref{conj:bdd mldk, Q-Gor, N-comp} (the difference is that we don't require the lc place of the complement to be of plt type over the singularity). Note that it is a special case of the uniform boundedness conjecture of mlds \cite{HLL-ACC-mld-3fold}*{Conjecture 7.2}, which has been verified in dimension two \cite{HL-surface-mld}.

\begin{conj} \label{conj:bdd mld for lc blowup}
Let $n,N\in\bN^*$ and let $\varepsilon>0$. Then there exists some $A_0=A_0(n,N,\varepsilon)$ such that for any $N$-complement $D$ of an $n$-dimensional klt variety $X$ and any lc center $W$ of $(X,D)$ such that $\hvol(\eta,X)\ge \varepsilon$ $($where $\eta$ is the generic point of $W)$, there exists a divisor $E$ over $\eta\in X$ such that
\[
A_{X,D}(E)=0, \quad\text{and}\quad A_X(E)\le A_0.
\]
\end{conj}

The other part is Conjecture \ref{conj:bdd mldk, Q-Gor, N-comp} for reduced complement. More precisely:  

\begin{lem} \label{lem:reduce to 1-comp}
Fix the dimension $n$. Assume that Conjecture \ref{conj:bdd mld for lc blowup} holds and that Conjecture \ref{conj:bdd mldk, Q-Gor, N-comp} holds when the complement $D$ is reduced $($i.e. all its coefficients are $1)$. Then Conjecture \ref{conj:bdd mldk, Q-Gor, N-comp} $($equivalently: Conjecture \ref{conj:bdd for mldK, vol>epsilon}$)$ holds. 
\end{lem}

\begin{proof}
We use the notation and assumptions in Conjecture \ref{conj:bdd mld for lc blowup}. Let $E$ be a divisor over $X$ with center $W$ such that $E$ is an lc place of $(X,D)$ and $A_X(E)\le A_0:=A_0(n,N,\varepsilon)$. We may assume that $A_0\ge 1$. By Lemma \ref{lem:BCHM extract divisor}, there exists a birational morphism $\pi\colon Y\to X$ such that $E$ is the unique exceptional divisor and $-E$ is ample. Since $E$ is an lc place of $(X,D)$ we have $\pi^*(K_X+D)=K_Y+D_Y+E\ge K_Y+E$, hence $(Y,E)$ is lc and any lc place of $(Y,E)$ is also an lc place of $(X,D)$. By Lemma \ref{lem:vol comparison lc blowup}, we have $\hvol(y,Y)\ge A_0^{-n}\varepsilon$ for all $y\in Y$. If $(Y,E)$ is already plt then there is nothing to prove. Otherwise let $Z\subseteq Y$ be a minimal lc center of $(Y,E)$ and let $\xi$ be the generic point of $Z$. Note that $\pi(\xi)=\eta$. Since Conjecture \ref{conj:bdd mldk, Q-Gor, N-comp} holds for $\xi\in Y$ with the reduced complement $E$ by assumption, and $\hvol(\xi,Y)\ge A_0^{-n}\varepsilon$, we see that there exists some Koll\'ar component $F$ over $\xi\in Y$ such that $A_{Y,E}(F)=0$ and $A_Y(F)\le A_1$ for some constant $A_1=A_1(n,A_0,\varepsilon)=A_1(n,N,\varepsilon)>0$. In particular, $F$ is an lc place of $(Y,E)$ and hence is also an lc place of $(X,D)$. Since $\pi^*K_X=K_Y+(1-A_X(E))E$, it follows that $A_X(F) = A_Y(F)+(A_X(E)-1)\cdot \ord_F(E)=A_X(E)\cdot A_Y(F)\le A_0 A_1$. Since $Z$ is a minimal lc center of $(Y,E)$, Lemma \ref{lem:closure of plt blowup over min lc center} implies that $F$ is of plt type over $(Y,E)$. By Lemma \ref{lem:kc descend along lc blowup}, $F$ is a Koll\'ar component over $X$. Thus Conjecture \ref{conj:bdd mldk, Q-Gor, N-comp} holds with $A(n,N,\varepsilon)=A_0 A_1$. 
\end{proof}

\section{Boundedness of \texorpdfstring{$\mldk$}{mldK}}

In this section, we study Conjecture \ref{conj:bdd for mldK} in codimension two, as well as Conjecture \ref{conj:bdd for mldK, vol>epsilon} (equivalently: Conjecture \ref{conj:bdd mldk, Q-Gor, N-comp}) in dimension $3$. In particular, we confirm the special boundedness of threefold klt singularities whose volumes are bounded from below. As a preliminary step, we first show that Conjecture \ref{conj:bdd mldk, Q-Gor, N-comp} holds for singularities $x\in X$ that belong to an analytically bounded family.

\begin{lem} \label{lem:mldk bdd, bdd germ}
Let $N$ be a positive integer and let $B\subseteq \cX\to B$ be a $\bQ$-Gorenstein family of klt singularities. Then there exists some $A>0$ depending only on $N$ and the family $B\subseteq \cX\to B$ such that for any klt singularity $x\in X$ with $(x\in X^{\an})\in (B\subseteq \cX^{\an}\to B)$ and any $N$-complement $D$ of $x\in X$, there exists some Koll\'ar component $E$ over $x\in X$ such that
\[
A_{X,D}(E)=0, \quad\text{and}\quad A_X(E)\le A.
\]
\end{lem}

\begin{proof}
By assumption, $\Coef(D)\subseteq \frac{1}{N}\bN\cap [0,1]$. Thus by Lemma \ref{lem:every kc is lc place} there also exists some $\varepsilon_0>0$ depending only on $N$ and $n=\dim X$ such that every Koll\'ar component $E$ over $x\in (X,(1-\varepsilon_0)D)$ is automatically an lc place of $(X,D)$ and in particular we have
\[
A_{X,(1-\varepsilon_0)D}(E)=A_X(E)-(1-\varepsilon_0)\ord_E(D)=A_X(E)-(1-\varepsilon_0)A_X(E)=\varepsilon_0 A_X(E).
\]
Therefore, to prove the lemma, it suffices to show that $\mldk(x,X,(1-\varepsilon_0)D)$ is bounded from above by some constant that only depends on $N$ and $B\subseteq \cX\to B$. 

Since the volume function is constructible \cite{Xu-quasi-monomial}*{Theorem 1.3}, there exists some $\varepsilon>0$ depending only on the family $B\subseteq \cX\to B$ such that $\hvol(x,X)\ge \varepsilon$. By Lemma \ref{lem:vol inequality}, we then deduce that $\hvol(x,X,(1-\varepsilon_0)D)\ge \varepsilon_0^n \hvol(x,X)\ge \varepsilon_0^n \varepsilon$. By \cite{HLQ-ACC}*{Theorem 1.7} and Lemma \ref{lem:mldk bdd if delta-plt blowup exist}, this further implies that there exists some $A=A(n,\varepsilon_0,\varepsilon)>0$ such that $\mldk(x,X,(1-\varepsilon_0)D)\le A$. Tracing back the proof, we see that the constant $A$ only depends on $N$ and the family $B\subseteq \cX\to B$. 
\end{proof}

\subsection{Codimension two case}

We next prove Conjecture \ref{conj:bdd for mldK} in codimension $2$ when $\bar{I}\subseteq \bQ$. This result is also needed when we prove Conjecture \ref{conj:bdd for mldK, vol>epsilon} for threefold singularities.

After localizing at the codimension two point, we immediately reduce to the surface case, albeit over a non-algebraically closed field. Thus, throughout this subsection, all the objects (singularities, divisors, etc.) we consider are defined over a field $\bk$ of characteristic $0$ that is not necessarily algebraically closed. The main technical result is the following:

\begin{prop} \label{prop:mldk bdd, dim=2, non-closed field}
Let $N\in\bN^*$ and let $x\in X$ be a smooth surface germ. Then there exists some constant $A_1>0$ depending only on $N$ such that for any $G\subseteq \Aut(x\in X)$ and any $G$-invariant $N$-complement $D$ of $x\in X$, there exists a $G$-invariant divisor $E$ $($defined over $\bk)$ over $x\in X$ with $A_{X,D}(E)=0$ and $A_X(E)\le A_1$.
\end{prop}

Note that when $\bk=\bar{\bk}$ and $G$ is trivial, this is already given by Lemma \ref{lem:mldk bdd, bdd germ}, so it remains to check that the corresponding Koll\'ar component can be chosen so that it is $G$-invariant and descends to the base field $\bk$. The key is the following uniqueness result.
 
\begin{lem} \label{lem:mld minimizer unique dim=2}
Let $x\in X$ be a smooth surface germ over $\bbk$ and let $D$ be a $\bQ$-complement of $x\in X$. Then the log discrepancy $A_X(E)$ of divisorial lc places of $x\in (X,D)$ is minimized by a unique divisor $E$. In particular, if $x\in (X,D)$ is defined over $\bk$ then $E$ is also defined over $\bk$ and it is invariant under $\Aut(x\in (X,D))$.
\end{lem}

\begin{proof}
Every divisor $E$ over $x\in X$ can be extracted via successive blowups of its centers (in particular, the sequence of blowups is canonical):
\begin{equation} \label{eq:canonical seq extract E}
    E\subseteq X_m\to \cdots\to X_1:=\Bl_x X\to X_0:=X.
\end{equation}
For any such sequence, if $E_i$ is the exceptional divisor of $X_i\to X_{i-1}$, then $C_{X_i}(E_{i+1})\in E_i$ and hence $A_X(E_{i+1})>A_X(E_i)$ for all $i$. We denote by $\ell(E):=m$ the number of blowups in this sequence and let
\[
\ell(x,X,D):=\min\{\ell(E)\,|\,C_X(E)=\{x\}, E \mbox{ is an lc place of } (X,D)\}.
\]

We say that a divisor $E$ over $x\in X$ dominates another $E'$, if $E'$ appears as an exceptional divisor in the canonical sequence \eqref{eq:canonical seq extract E} of blowups associated to $E$. By the above discussion, $A_X(E)>A_X(E')$ if $E$ dominates $E'$ and $E\neq E'$. Thus it suffices to show that there exists a divisor over $x\in (X,D)$ that is an lc place of $(X,D)$ and that at the same time is dominated by all other divisorial lc place of $(X,D)$.

We prove this stronger statement by induction on $\ell(x,X,D)$. Let $E$ be a divisorial lc place of $(X,D)$ that computes $\ell(x,X,D)$, i.e. $\ell(E)=\ell(x,X,D)$. Let $E_1$ be the exceptional divisor of the ordinary blowup $\pi_1\colon X_1=\Bl_x X\to X$. When $\ell(E)=1$, we have $E=E_1$ and it is dominated by all divisorial lc place of $(X,D)$, since all blowup sequence has to begin with the ordinary blowup. Thus the statement holds when $\ell(x,X,D)=1$. If $\ell(E)>1$, then $E_1$ is not an lc place of $(X,D)$. Since $x$ is an lc center of $(X,D)$, we have $\mult_x D>1$ (otherwise $(X,D)$ is canonical by \cite{KM98}*{Theorem 4.5} and the lc centers would have dimension one) and $\pi_1^*(K_X+D)=K_{X_1}+D_1$, where $D_1=\pi_{1*}^{-1}D+(\mult_x D-1)E_1\ge 0$. By assumption $E_1$ is not an lc place of $(X_1,D_1)$, thus the lc center $W$ of $(X_1,D_1)$ intersect $E_1$ at a finite number of points. By Koll\'ar-Shokurov connectedness \cite{KM98}*{Theorem 5.48}, we deduce that $W\cap E=\{x_1\}$ consists of a single point (as a set). It follows that $\ell(x,X,D)=\ell(x_1,X_1,D_1)+1$ and all lc places of $x\in (X,D)$ dominates the ordinary blowup of $x_1$ on $X_1$. The result now follows by induction.
\end{proof}

\begin{proof}[Proof of Proposition \ref{prop:mldk bdd, dim=2, non-closed field}]
Let $(X_{\bbk},D_{\bbk}):=(X,D)\times_\bk \bbk$. By Lemma \ref{lem:mldk bdd, bdd germ}, there exists some constant $A_1>0$ depending only on $N$ and some divisorial lc place $F$ of $x\in (X_{\bbk},D_{\bbk})$ such that $A_{X_{\bbk}}(F)\le A_1$. By Lemma \ref{lem:mld minimizer unique dim=2}, the log discrepancy of divisorial lc places of $x\in (X_{\bbk},D_{\bbk})$ is minimized by some $G$-invariant divisor $E$ that's defined over $\bk$. In particular, $A_X(E)\le A_1$. 
This proves the statement of the proposition. 
\end{proof}

We are now ready to prove a version of Conjecture \ref{conj:bdd for mldK} at codimension two points.

\begin{prop} \label{prop:bdd mldk, codim=2}
Let $I=\bar{I}\subseteq [0,1]\cap \bQ$ be a DCC set. Then there exists some constant $A$ depending only on $I$ such that 
\[
\mldk(\eta,X,\Delta)\le A
\]
for any klt pair $(X,\Delta)$ with $\Coef(\Delta)\subseteq I$ and any codimension two point $\eta\in X$.
\end{prop}

\begin{proof}
After localizing at $\eta\in X$, we may assume that $\eta$ is a closed point $x$ and $x\in (X,\Delta)$ is a surface klt singularity (over a field $\bk$ that's not necessarily algebraically closed). By Lemma \ref{lem:strict N-complement}, there exists an $N$-complement $D$ of $x\in (X,\Delta)$ for some integer $N$ that only depends on the coefficient set $I$. Let $\pi\colon (\tx\in \tX)\to (x\in X)$ be the universal cover. Let $\tD=\pi^*D$ and $\tDelta=\pi^*\Delta$. Then $\tx\in \tX$ is a smooth surface germ and $\tD+\tDelta$ is an $N$-complement of $\tx\in\tX$ by Lemma \ref{lem:strict comp quasi-etale pullback}. 

First consider a special case, i.e. when $D$ is reduced. If $(X,D)$ is not plt, then $x$ is an lc center of $(X,D)$ and $\Delta=0$. In this case $\tD$ is a $1$-complement of $\tx\in \tX$, thus it is nodal and the exceptional divisor $\tE$ of the ordinary blowup of $\tx$ is an lc place of $(\tX,\tD)$. Clearly $\tE$ descends to a Koll\'ar component $E$ over $x\in X$ and $A_X(E)\le A_{\tX}(\tE)\le 2$ by \cite{KM98}*{Proof of Proposition 5.20}. If $(X,D)$ is plt, then $\tD$ is smooth, $X=\tX/\mu_m$ for some $m\in\bN$, and we can choose coordinates $u,v$ at $\tx\in \tX$ that diagonalize the $\mu_m$-action and with $\tD=(u=0)$. Since $\tx$ is an lc center of $(\tX,\tDelta+\tD)$, by adjunction we see that $(\tD,\tDelta|_{\tD})$ is strictly lc, hence $\tDelta|_{\tD}=\tx$. Since $\tDelta+\tD$ is an $N$-complement of $\tx\in \tX$ and in particular $N\tDelta$ has integer coefficients, we deduced from $\tDelta|_{\tD}=\tx$ that the local defining equation of $N\tDelta$ has the form $v^N+u\phi(u,v)=0$ for some $\phi\in\bk[[u,v]]$. Note that $\phi\neq 0$ since $(\tX,\tDelta)$ is klt. 

We seek a weighted blowup up with $\wt(u)=a$, $\wt(v)=b$ that provides a Koll\'ar component over $\tx\in (\tX,\tDelta)$ which is also an lc place of $(\tX,\tDelta+\tD)$. For this we look at the Newton polygon (denoted as $Q$) of $v^N+u\phi(u,v)$, i.e. the convex hull of the union of $(p,q)+\bR_{\ge 0}^2$ where $(p,q)\in \bN^2$ varies among pairs for which $u^p v^q$ has non-zero coefficients in $v^N+u\phi(u,v)$. Certainly $(0,N)\in Q$. Let $(r,s)\in Q$ be the vertex that lies on the same edge as $(0,N)$. We choose the weights $a,b$ so that $\gcd(a,b)=1$, $ar+bs=bN$. In particular, $a\le N-s\le N$. It is not hard to check that the corresponding exceptional divisor $\tE$ is a a Koll\'ar component over $\tx\in (\tX,\tDelta)$ (essentially because it induces a degeneration of $\tx\in (\tX,\tDelta)$ to the klt singularity $\tx\in (\tX,\frac{1}{N}(v^N+u^r v^s=0))$, c.f. Lemma \ref{lem:wt blowup kc criterion}). Since the coordinates $u,v$ diagonalize the $\mu_m$-action, the divisor $\tE$ is $\mu_m$-invariant, hence it descends to a Koll\'ar component over $x\in (X,\Delta)$ by \cite{LX-stability-kc}*{Lemma 2.13}. By \cite{KM98}*{Proposition 5.20} and a direct calculation, we then have $A_{X,\Delta}(E)\le A_{\tX,\tDelta}(\tE)=a\le N$. As $N$ only depends on $I$, we conclude the proof in the case when the $\bQ$-complement $D$ is reduced.

In the general case, we apply Proposition \ref{prop:mldk bdd, dim=2, non-closed field} to get a constant $A_1$ that only depends on $N$ (thus it depends on $I$ only), and an $\Aut(\tx\in \tX)$-invariant divisor $\tE$ over $\tx\in\tX$ such that $A_{\tX,\tDelta+\tD}(\tE)=0$ and $A_{\tX}(\tE)\le A_1$. By \cite{KM98}*{Proposition 5.20}, the divisor $\tE$ induces a divisor $E$ over $x\in X$ that is an lc place of $(X,\Delta+D)$ and still satisfies $A_X(E)\le A_1$. In particular, $A_{X,\Delta}(E)\le A_X(E)\le A_1$ and $E$ is of lc type over $(X,\Delta)$. If $E$ is already a Koll\'ar component then we are done. Otherwise, if $Y\to X$ is the prime blowup of $E$ then the minimal lc centers of $(Y,\Delta_Y+E)$ are $0$-dimensional and supported on $\lfloor \Diff_E(\Delta_Y) \rfloor$. But as $-(K_Y+\Delta_Y+E)|_E=-(K_E+\Diff_E(\Delta_Y))$ is ample, we see that $E\cong \bP^1$ and $\deg \lfloor \Diff_E(\Delta_Y) \rfloor =1$, hence the minimal lc center of $(Y,\Delta_Y+E)$ is a $\bk$-rational closed point $y\in Y$. Note that $E$ is a reduced $N$-complement of $(Y,\Delta_Y)$ thus by the special case treated above we may choose a Koll\'ar component $E$ over $y\in (Y,\Delta_Y)$ such that $A_{Y,\Delta_Y}(E)\le A_0$ for some constant $A_0$ that only depends on $I$. By Lemma \ref{lem:kc descend along lc blowup}, we know that $E$ is also a Koll\'ar component over $x\in (X,\Delta)$. As in the proof of Lemma \ref{lem:reduce to 1-comp}, we also have $A_{X,\Delta}(E)\le A_0 A_1$. Since both constants $A_0, A_1$ only depend on $I$, we are done.
\end{proof}

\begin{cor} \label{cor:bdd mldk, N-comp, codim=2}
Conjecture \ref{conj:bdd mldk, Q-Gor, N-comp} holds at codimension $2$ points.
\end{cor}

\begin{proof}
This is an immediate by Proposition \ref{prop:bdd mldk, codim=2} and the proof of Proposition \ref{prop:reduce to N-complement}.
\end{proof}

\subsection{Threefold singularities}

In this subsection, we will prove Conjecture \ref{conj:bdd for mldK, vol>epsilon} in dimension $3$. Note that it suffices to consider the case when $\eta\in X$ is a closed point, since we already prove Conjecture \ref{conj:bdd for mldK} at codimension two points. According to Proposition \ref{prop:reduce to N-complement} and Lemma \ref{lem:reduce to 1-comp}, the first step is to verify Conjecture \ref{conj:bdd mld for lc blowup} for $3$-dimensional klt singularities.

\begin{prop} \label{prop:mld of lc blowup bdd, dim=3}
Let $N\in\bN^*$ and let $\varepsilon>0$. Then there exists some $A_0=A_0(3,N,\varepsilon)$ such that for any $3$-dimensional klt singularity $x\in X$ with $\hvol(x,X)\ge \varepsilon$ and any $N$-complement $D$ of $x\in X$, there exists some divisor $E$ over $x\in X$ such that
\[
A_{X,D}(E)=0, \quad\text{and}\quad A_X(E)\le A_0.
\]
In other words, Conjecture \ref{conj:bdd mld for lc blowup} holds in dimension $3$.
\end{prop}

We first consider the terminal singularity case. Eventually we will reduce the proof to this special case.

\begin{lem} \label{lem:bdd lc mld terminal case, dim=3}
Let $N\in\bN^*$ and let $\varepsilon>0$. Then there exists some $A_1=A_1(3,N,\varepsilon)$ such that for any $3$-dimensional terminal singularity $x\in X$ with $\hvol(x,X)\ge \varepsilon$ and any $N$-complement $D$ of $x\in X$, there exists some divisor $E$ over $x\in X$ such that
\[
A_{X,D}(E)=0, \quad\text{and}\quad A_X(E)\le A_1.
\]
\end{lem}

\begin{proof}
Certainly the plan is to use the classification of terminal singularities \cite{Mori-terminal-classification} (see also \cite{Reid-young-person-guide}*{Section 6}). Assume that $x\in X$ has index $r$ and let $p\colon (\tx\in \tX)\to (x\in X)$ be the index $1$ cover. By \cite{XZ-uniqueness}*{Theorem 1.3} and Lemma \ref{lem:index bound via volume}, we have $r\le \frac{27}{\varepsilon}$ and $\hvol(\tx,\tX)=r\cdot \hvol(x,X)\ge \varepsilon$. Let $\tD=p^*D$. By assumption, $N\tD$ has integer coefficient; another application of Lemma \ref{lem:index bound via volume} then gives an integer $N_1$ depending only on $N,\varepsilon$ such that $N_1 \tD$ is Cartier. By \cite{Reid-young-person-guide}*{Theorem in (6.1)}, there exists a linear $\mu_r$-action on $\bA^{4}$ and an eigenfunction $f\in \cO_{\bA^4,0}$ such that analytically locally $x\in X$ is isomorphic to $0\in (f=0)/\mu_r$. Since the statement of the lemma is analytically local, we may assume that $(x\in X)=\left(0\in (f=0)/\mu_r \right)$. In particular $\tX\cong (f=0)\subseteq \bA^4$. Since $N_1 \tD$ is Cartier and $\mu_r$-invariant, we may also write $N_1 \tD = (g=0)|_{\tX}$ for some eigenfunction $g\in \cO_{\bA^4,0}$. 

Let $\widetilde{\Delta}=(g=0)\subseteq \bA^4$ and consider the pair $y\in (Y,\Delta):=0\in (\bA^4,\tX+\frac{1}{N_1}\widetilde{\Delta})/\mu_r$, namely $Y=\bA^4 /\mu_r$ and if $q\colon \bA^4\to Y$ is the quotient map then the divisor $\Delta$ is chosen such that $K_{\bA^4}+\tX+\frac{1}{N_1}\widetilde{\Delta}=q^*(K_Y+\Delta)$. By Lemma \ref{lem:strict comp quasi-etale pullback}, $\tD$ is a $\bQ$-complement of $\tx\in \tX$. By inversion of adjunction (see e.g. \cite{Kol13}*{Theorem 4.9}), this implies that $(\bA^4,\tX+\frac{1}{N_1}\tDelta)$ is lc in a neighbourhood of $\tX$. Since $\tX$ is clearly an lc center of this pair, and the origin is an lc center of $(\tX,\tD)$, we see that the minimal lc center of $(\bA^4,\tX+\frac{1}{N_1}\tDelta)$ is the origin. Hence $\tX+\frac{1}{N_1}\tDelta$ is a $\bQ$-complement of $0\in \bA^4$, and $\Delta$ is a $\bQ$-complement of $y\in Y$ (again by Lemma \ref{lem:strict comp quasi-etale pullback}). 

Note that $\Coef(\Delta)\subseteq \frac{1}{rN_1}\bZ \cap [0,1]$ and $\Delta=X+\Gamma$ for some effective divisor $\Gamma$ such that $\Gamma|_X=D$. Since $r\le \frac{27}{\varepsilon}$ and for each $r\in \bN$ there are only finitely many non-isomorphic linear actions of $\mu_r$ on $\bA^4$, the singularities $y\in Y$ that arise from the above construction belong to a bounded family, hence by Lemma \ref{lem:mldk bdd, bdd germ} we deduce that there exist some constant $A_1=A_1(r,N_1)=A_1(\varepsilon,N)>0$ and some Koll\'ar component $F$ over $y\in Y$ that is also an lc place of $(Y,\Delta)$ such that $A_{Y}(F)\le A_1$. Let $\pi\colon Y'\to Y$ be the associated plt blowup and let $X',\Gamma'$ be the strict transform of $X,\Gamma$ on $Y'$. Let $E$ be an irreducible component of $F|_{X'}$. By adjunction, $E$ is an lc place of $(X,\Gamma|_X)=(X,D)$. We claim that $E$ is the sought divisor. Clearly $E$ is centered at $x\in X$. Since $F$ is an lc place of $(Y,\Delta)$ we have $K_{Y'}+F+X'+\Gamma'=\pi^*(K_Y+\Delta)$ and the pair $(Y',F+X'+\Gamma')$ is lc. By the classification of lc surface pairs (see e.g. \cite{Kol13}*{Example 3.28}), it follows that $X'$ is normal at the generic point of $E$ and $\mult_E(F|_{X'})\le 1$. 
We also have $K_{Y'}+F+X'=\pi^*(K_Y+X)+A_{Y,X}(F)\cdot F$ which by adjunction gives $A_X(E)=A_{Y,X}(F)\cdot \mult_E(F|_{X'})\le A_Y(F)\le A_1$. This completes the proof.
\end{proof}

We now prove the general case of Proposition \ref{prop:mld of lc blowup bdd, dim=3} by reducing it to the terminal case. 

\begin{proof}[Proof of Proposition \ref{prop:mld of lc blowup bdd, dim=3}]
Let $\pi\colon Y\to X$ be a terminal modification of $x\in X$ (whose existence is given by e.g. \cite{Kol13}*{Theorem 1.33}). It has the property that $\pi^*K_X=K_Y+\Gamma$ for some effective exceptional divisor $\Gamma$, $Y$ is $\bQ$-Gorenstein and has only terminal singularities. Moreover, by Lemma \ref{lem:vol under birational map} we have $\hvol(y,Y)\ge \hvol(x,X)\ge \varepsilon$ for all $y\in Y$. Let $D_Y=\Gamma+\pi^*D$. Since $K_Y+D_Y=\pi^*(K_X+D)$ and $D$ is an $N$-complement of $x\in X$, we see that $D_Y$ is an $N$-complement of $Y$, the pair $(Y,D_Y)$ is strictly log canonical, and it has a minimal lc center $W$ that's contained in $\pi^{-1}(x)$. If $\dim W=2$, then we may take $E=W$ as $A_X(W)=A_{Y,\Gamma}(W)\le 1$. If $\dim W=1$ then it has codimension $2$ in $Y$. By Corollary \ref{cor:bdd mldk, N-comp, codim=2}, there exist some constant $A_1>0$ depending only on $N$ and some divisor $E$ over $x\in X$ (in fact it has center $W$ on $Y$) such that $A_{X,D}(E)=A_{Y,D_Y}(E)=0$ and $A_X(E)=A_{Y,\Gamma}(E)\le A_Y(E)\le A_1$. Finally if $\dim W=0$ then $W=\{y\}$ for some $y\in Y$. Since $y\in Y$ has terminal singularity, we know by Lemma \ref{lem:bdd lc mld terminal case, dim=3} that there exist another constant $A_2>0$ (depending only on $N$ and $\varepsilon$) and a divisor $E$ over $y\in Y$ such that $A_{X,D}(E)=A_{Y,D_Y}(E)=0$ and $A_X(E)=A_{Y,\Gamma}(E)\le A_Y(E)\le A_2$. Therefore, the proposition holds with $A_0=\max\{1,A_1,A_2\}$. 
\end{proof}

Going back to the proof of Conjecture \ref{conj:bdd for mldK, vol>epsilon}, the second step is to verify Conjecture \ref{conj:bdd mldk, Q-Gor, N-comp} for reduced $\bQ$-complements.

\begin{prop} \label{prop:mldk bdd, dim=3, 1-comp}
Let $\varepsilon>0$. Then there exists some constant $A>0$ depending only on $\varepsilon$ such that for any $3$-dimensional klt singularity $x\in X$ with $\hvol(x,X)\ge \varepsilon$ and any reduced $\bQ$-complement $D$, there exists some Koll\'ar component $E$ over $x\in X$ such that 
\[
A_{X,D}(E)=0, \quad\text{and}\quad A_X(E)\le A.
\]
In other words, Conjecture \ref{conj:bdd mldk, Q-Gor, N-comp} holds for reduced $\bQ$-complements in dimension $3$.
\end{prop}

Again we first treat a special case. Recall that a singularity $x\in X$ is called a cDV singularity if by taking general complete intersections of hypersurfaces that contains $x$ one gets a Du Val singularity.

\begin{lem} \label{lem:mldk bdd, dim=3, 1-comp, terminal}
Proposition \ref{prop:mldk bdd, dim=3, 1-comp} holds when the simultaneous index one cover of $K_X$ and $D$ $($see Section \ref{sect:index one covers}$)$ is a cDV singularity. 
\end{lem}

\begin{proof}
Let $\tx\in\tX$ be the simultaneous index one cover of $K_X$ and $D$ so that $(x\in X)\cong (\tx\in \tX)/H$ for some finite abelian group $H$. By Lemma \ref{lem:index bound via volume}, we have $|H|\le \frac{27}{\varepsilon}$. Let $\tD\subseteq \tX$ be the preimage of $D$. 
Then $\tD$ is a Cartier divisor and we may write $\tD=(g=0)$ for some $g\in \fm$ (where $\fm\subseteq \cO_{\tX,\tx}$ denotes the maximal ideal). We divide the proof into the following three cases:
\begin{enumerate}
    \item $\tx\in \tX$ is a smooth point.
    \item $\tx\in \tX$ is singular and $g\in \fm^2$.
    \item $\tx\in \tX$ is singular and $g\not\in \fm^2$.
\end{enumerate}

\begin{case}
$\tx\in \tX$ is a smooth point. Then analytically $(x\in X)\cong (0\in \bA^3)/H$ belongs to a bounded family of singularities (since the size of $H$ is bounded from above and for each choice of $H$ there are only finitely many non-isomorphic actions of $H$ on $0\in \bA^3$). By Lemma \ref{lem:mldk bdd, bdd germ}, we see that there exists some constant $A_1>0$ that only depends on $\varepsilon$ (since the upper bound of the index $r$ only relies on $\varepsilon$) and some Koll\'ar component $E$ over $x\in X$ such that $A_{X,D}(E)=0$ and $A_X(E)\le A_1$. Thus the lemma holds in this case.
\end{case}

\begin{case} \label{case:mult(D)>=2}
$\tx\in \tX$ is singular and $g\in \fm^2$. 
Since $\tx\in \tX$ is a cDV singularity, analytically we may identify it with $0\in (f=0)\subseteq \bA^4$ for some $f\in \hat{\cO}_{\bA^4,0}$ with $\mult(f)=2$. Moreover, since $H$ is abelian, we can assume that it acts diagonally on the coordinates $y,z,w,t$ of $\bA^4$, all monomials in $f$ have the same eigenvalue, and the isomorphism $\tX\cong (f=0)$ is $H$-equivariant. Since $(\tX,\tD)$ is lc, we also know that $(\tX,\fm^2)$ is lc and any lc place of $(\tX,\fm^2)$ is also an lc place of $(\tX,\tD)$. Therefore, it suffices to find an $H$-invariant lc place of $(\tX,\fm^2)$ with bounded log discrepancy that is a Koll\'ar component. By adjunction \cite{Kol13}*{Theorem 4.9}, we see that a general plane section of $X$ has semi-log canonical singularities, hence is a nodal curve. It follows that the quadratic part in the Taylor expansion of $f$ has rank at least $2$.

First assume that the rank is at least $3$. Let $\tE$ be the exceptional divisor of the ordinary blowup of $\tx\in\tX$. It is straightforward to check that $\tE$ is an lc place of $(\tX,\fm^2)$, $A_{\tX}(\tE)=2$, and $\tE$ is a Koll\'ar component ($\tE$ is a quadric surface and in particular has klt singularities). Clearly $\tE$ is $H$-invariant (in fact it is invariant under $\Aut(\tx\in\tX)$), thus by \cite{LX-stability-kc}*{Lemma 2.13} it descends to a Koll\'ar component $E$ over $x\in X$ with $A_{X,D}(E)=0$ and $A_X(E)\le 2$. Hence the lemma holds in this case.

Assume next that the rank of the quadratic part of $f$ is two. We claim that there exists a finite group $G\subseteq \Aut(\tx\in \tX)$ of order at most $2|H|^4$ containing $H$ such that $W:=\tX/G$ belongs to an analytically bounded family of singularities (that depends only on $\varepsilon$). If $f$ contains the monomial $t^2$, then after an $H$-equivariant change of variables (c.f. the proof of \cite{Reid-young-person-guide}*{Page 395, Proposition}, especially \cite{Reid-young-person-guide}*{Page 394, Rule III}), we may assume that $f=t^2+h(y,z,w)$ for some $\mult(h)\ge 2$. Let $\tau$ be the involution $(y,z,w,t)\mapsto (y,z,w,-t)$ and let $G\subseteq \Aut(\tx\in \tX)$ be the (abelian) subgroup generated by $H$ and $\tau$. Clearly $|G|\le 2|H|$ and $\tX/\tau$ is a smooth point, thus $\tX/G=(\tX/\tau)/(G/\tau)$ is a quotient singularity. Since $|G/\tau|\le |H|\le \frac{27}{\varepsilon}$ we see that there are only finitely many non-isomorphic quotients as before and hence $W=\tX/G$ belongs to an analytically bounded family of singularities. 

If $f$ does not contain any of $y^2,z^2,w^2,t^2$, then as its quadratic part has rank $2$ we have $f=yz+h(w,t)$ after a change of variable (again by \cite{Reid-young-person-guide}*{Page 395, Proposition}). This time let $\tau$ be the involution $(y,z,w,t)\mapsto (z,y,w,t)$ and let $G\subseteq \Aut(\tx\in \tX)$ be the subgroup generated by $H$ and $\tau$. It is not hard to see that for every $\varphi\in G$ either $\varphi$ or $\tau\varphi$ acts diagonally on $(y,z,w,t)$ with order at most $|H|$, hence $|G|\le 2|H|^4$. Let $G_0<G$ be the subgroup generated by all the conjugates of $\tau$; it is also the smallest normal subgroup that contains $\tau$. By a theorem of Chevalley, $\bA^4/G_0\cong \bA^4$ and in fact it is not hard to check that the quotient map is given by $(y,z,w,t)\mapsto (y^{r_0}+z^{r_0},yz,w,t)$ for some $r_0\in\bN^*$. It then follows from the equation of $f$ that $(\tx\in \tX)/G_0$ is a smooth point and $\tX/G=(\tX/G_0)/(G/G_0)$ is a quotient singularity. As $|G/G_0|$ is at most $|H|^4\le \frac{27^4}{\varepsilon^4}$, we conclude as before that there are only finitely many non-isomorphic quotients and $W=\tX/G$ belongs to an analytically bounded family of singularities which only depends on $\varepsilon$. This proves the claim.

Let $(W,\Delta_W+\fa)$ be the crepant $G$-quotient of $(\tX,\fm^2)$, where $\fa$ is a $\bQ$-ideal. Note that $(\tX,\fm^2)$ is strictly lc, thus the same holds for $(W,\Delta_W+\fa)$. By construction the coefficients of $\Delta_W+\fa$ lie in $\frac{1}{|G|}\bN$. By Lemma \ref{lem:mldk bdd, bdd germ}, there exist some constant $A'_2>0$ depending only on $\varepsilon$ and a Koll\'ar component $\tE$ over $W$ that is an lc place of $(W,\Delta_W+\fa)$ such that $A_{W}(\tE)\le A'_2$. By \cite{LX-stability-kc}*{Lemma 2.13}, it pulls back to a $G$-invariant Koll\'ar component over $\tX$ and hence also induces a Koll\'ar component $E$ over $X$. By \cite{KM98}*{Proposition 5.20} we know that $E$ is an lc place of $(X,D)$ and $A_X(E)\le |G|\cdot A_{W,\Delta_W}(\tE)\le A_2$ for some constant $A_2>0$ that only depends on $\varepsilon$. Therefore we are done in this case.
\end{case}

\begin{case}
$\tx\in \tX$ is singular and $g\not\in \fm^2$. Then up to a change of coordinate we may assume that $g=t$ and $\tX\cong (f=0)\subseteq \bA^4$ with $\mult(f)=2$. We may write $f=t f_1(y,z,w,t)+h(y,z,w)$ and in particular $\tD\cong (h=0)\subseteq \bA^3$. By adjunction, the surface $\tD$ has slc singularity and the origin is an lc center. Thus $2\le \mult(h)\le 3$. 

Suppose first that $\mult(h)=3$. Let $h_1$ be the homogeneous term of degree $3$ in $h$. Let $q\colon \tD'\to \tD$ be the ordinary blowup of the origin. Then $q^*K_{\tD}=K_{\tD'}+F$ where $F\cong (h_1(y,z,w)=0)\subseteq \bP^2$ is the $q$-exceptional divisor. Since $\tD$ is slc, so is the pair $(\tD',F)$, thus by adjunction we know that the cubic curve $F$ is at most nodal (for us it is enough to know that $h_1$ is irreducible). Note that $\mult(f_1)=1$ since $\mult(f)=2$. If the linear term in $f_1$ is not proportional to $t$, then we may apply an $H$-equivariant change of coordinates between $y,z,w$ so that $f_1=ay+bt+(\mult\ge 2)$ ($a\neq 0$). Consider the weighted blowup with $\wt(y,z,w,t)=(1,1,1,2)$. By Lemma \ref{lem:wt blowup kc criterion}, the exceptional divisor $E$ is a Koll\'ar component, since the corresponding initial term $ayt+h_1(y,z,w)$ gives a $cA$-type singularity. Note that $E$ is an $H$-invariant lc place of $(X,D)$ and $A_X(E)=2$.

The other possibility is that the linear term in $f_1$ is proportional to $t$, i.e. $f_1=at+(\mult\ge 2)$ ($a\neq 0$). Consider the weighted blowup with $\wt(y,z,w,t)=(2,2,2,3)$. The corresponding initial term gives the hypersurface $(at^2+h_1(y,z,w)=0)\subseteq \bA^4$ which is of $cD$-type by Lemma \ref{lem:cD type criterion}. Thus the exceptional divisor $E$ of the weighted blowup is a Koll\'ar component by Lemma \ref{lem:wt blowup kc criterion}. It is also an $H$-invariant lc place of $(X,D)$ and $A_X(E)=3$.

Next suppose that $\mult(h)=2$. Then the quadratic part of $h$ has rank $1$, otherwise $\tD$ is either a union of two planes or has $A_m$ singularity for some $m\in\bN^*$, contradicting the assumption that the origin is an lc center of $\tD$. Thus we may choose coordinates so that $h=y^2+h_0(z,w)$ and $\mult(h_0)\ge 3$. By Weierstrass preparation theorem (applied to $f$), we then have $f=\unit\cdot (y^2+ay+b)$ for some $a,b\in (z,w,t)\subseteq \bC[[z,w,t]]$, which can be turned into $f=\unit\cdot (y^2+c(z,w,t))$ with another change of variable $y\mapsto y-\frac{a}{2}$ which leaves $t$ fixed. Here we need to further justify that our application of the Weierstrass preparation theorem is $H$-equivariant. Indeed, one way to prove the Weierstrass preparation theorem is through the Weierstrass division theorem: there exists some unique $s_1\in \bC[[y,z,w,t]]$ and $s_2\in \bC[[z,w,t]][y]$ such that $y^2=f\cdot s_1+s_2$ and the degree of $y$ in $s_2$ is $\le 1$; one then checks that $s_1$ is a unit and if $s_2=-ay-b$ then we get the desired form $f=\unit\cdot (y^2+ay+b)$. Since all monomials in $f$ have the same eigenvalues under the $H$-action, the same holds for both $s_1$ and $s_2$ by their uniqueness. Thus with a $H$-equivariant change of coordinates we may assume that $f$ has the form $f=y^2+c(z,w,t)$ and at the same time $g=t$. We can now argue as in Case \ref{case:mult(D)>=2}: if $\tau$ is the involution $(y,z,w,t)\mapsto (-y,z,w,t)$ then it leaves $\tD$ invariant and $\tX/\tau$ is a smooth point. Hence if $G\subseteq \Aut(\tx\in \tX)$ is the (abelian) subgroup generated by $H$ and $\tau$, then $|G/\tau|\le |H|\le \frac{27}{\varepsilon}$ and $W:=\tX/G=(\tX/\tau)/(G/\tau)$ is a quotient singularity and belongs to an analytically bounded family of singularities. The crepant $G$-quotient $(W,\Delta_W)$ of $(\tX,\tD)$ has coefficients in $\frac{1}{|G|}\bN$ and therefore as in Case \ref{case:mult(D)>=2} we get some constant $A_3$ that depends only on $\varepsilon$ and some Koll\'ar component $E$ over $x\in X$ that is an lc place of $(X,D)$ and satisfies $A_X(E)\le A_3$.
\end{case}

Putting the three cases together, we see that the proof of the lemma is finished by setting $A=\max\{3,A_1,A_2,A_3\}$.
\end{proof}

The following result is used in the above proof.

\begin{lem} \label{lem:cD type criterion}
Let $n\ge 3$ and let $f\in \bC[[x_2,\cdots,x_n]]$. Assume that $\mult(f)=3$ and let $f_3$ be the leading term of $f$. Then the hypersurface singularity $(x_1^2+f(x_2,\cdots,x_n)=0)\subseteq\bA^n$ is of $cD$-type if and only if its singular locus has codimension at least $2$ and $f_3$ is not a cube.
\end{lem}

\begin{proof}
The conditions are preserved by taking general hyperplane sections of the form $a_2 x_2+\cdots+a_n x_n=0$. Thus we may assume that $n=3$, where the statement follows from \cite{KM98}*{Step 4 on Page 126}.
\end{proof}

\begin{proof}[Proof of Proposition \ref{prop:mldk bdd, dim=3, 1-comp}]
Let $N=\lfloor \frac{27}{\varepsilon} \rfloor !$. By Lemma \ref{lem:index bound via volume}, we know that $D$ is an $N$-complement.
Let $A_1$ (resp. $A_2$) be the constant from  Corollary \ref{cor:bdd mldk, N-comp, codim=2} (resp. Lemma \ref{lem:mldk bdd, dim=3, 1-comp, terminal}). We will show that the proposition holds with $A=\max\{A_1,A_2\}$. We do this by induction on $d(x,X)$, the number of divisors $E$ over $x\in X$ such that $A_X(E)\le 1$. By \cite{KM98}*{Proposition 2.36(2)}, $d(x,X)<+\infty$. Let $\tx\in\tX$ be the simultaneous index one cover of $K_X$ and $D$ so that $(x\in X)\cong (\tx\in \tX)/H$ for some finite abelian group $H$. Let $\tD\subseteq \tX$ be the preimage of $D$. Then $\tx\in\tX$ is a Gorenstein canonical singularity and $\tD$ is Cartier. If $\tX$ is a cDV singularity, then the statement already follows from Lemma \ref{lem:mldk bdd, dim=3, 1-comp, terminal}. In particular, the proposition holds when $d(x,X)=0$ (i.e. when $x\in X$ is terminal), since the index one cover of a threefold terminal singularity is cDV. If $\tx\in\tX$ is not a cDV singularity, then by \cite{KM98}*{Lemma 5.30 and Theorem 5.35} there exists an $H$-invariant divisor $\tF$ over $\tx\in \tX$ with $A_{\tX}(\tF)=1$. Since $\tD$ is Cartier, we obtain $\ord_{\tF}(\tD)\ge 1$ and hence $A_{\tX,\tD}(\tF)\le 0$. As $(\tX,\tD)$ is lc, it follows that $\tF$ is an lc place of $(\tX,\tD)$. Let $F$ be the induced divisor over $X$. Then by \cite{KM98}*{Proposition 5.20} we have $A_X(F)\le 1$ and $A_{X,D}(F)=0$. By Lemma \ref{lem:BCHM extract divisor}, there exists a birational morphism $\pi\colon Y\to X$ with a unique exceptional divisor $F$ such that $-F$ is $\pi$-ample. Let $D_Y=\pi^{-1}_* D$. Since $F$ is an lc place of $(X,D)$, any lc place $E$ of $(Y,F)$ is also an lc place of $(X,D)$. Moreover, if $E$ is of plt type over $Y$, then it is a Koll\'ar component over $x\in X$ by Lemma \ref{lem:kc descend along lc blowup}. By construction, $K_Y\le \pi^*K_X$, hence $A_X(E)\le A_Y(E)$ for any divisor $E$ over $Y$. Thus it suffices to find an lc place $E$ of $(Y,F)$ that is of plt type over $Y$ and has log discrepancy $A_Y(E)\le A$. If the minimal lc center of $(Y,F)$ has codimension $\le 2$ this follows from Corollary \ref{cor:bdd mldk, N-comp, codim=2} and our choice of the constant $A$. Otherwise $F$ is a reduced $\bQ$-complement of $y\in Y$ for some $y\in \pi^{-1}(x)$. By Lemma \ref{lem:vol under birational map}, we have $\hvol(y,Y)\ge \varepsilon$. From the definition we also have $d(y,Y)\le d(x,X)-1$, hence in this case the result follows from the induction hypothesis and we are done.
\end{proof}

\begin{cor} \label{cor:bdd for mld^K, vol>epsilon, dim=3}
Conjecture \ref{conj:bdd for mldK, vol>epsilon} holds in dimension three.
\end{cor}

\begin{proof}
This is a combination of Proposition \ref{prop:bdd mldk, codim=2}, Lemma \ref{lem:reduce to 1-comp}, Proposition \ref{prop:mld of lc blowup bdd, dim=3} and Proposition \ref{prop:mldk bdd, dim=3, 1-comp}.
\end{proof}

\begin{cor} \label{cor:special bdd dim=3}
Conjecture \ref{conj:vol>epsilon implies bounded} holds in dimension three.
\end{cor}

\begin{proof}
This is a combination of Theorem \ref{thm:delta-plt blowup strong version} and Corollary \ref{cor:bdd for mld^K, vol>epsilon, dim=3}.
\end{proof}

\section{Examples and discussions}

In general one should be more careful when making analogy between $\mldk$ and the usual mld. In this section, we give some examples that illustrate the difference between them, and discuss some related questions. Our first example shows that $\mldk$ can be arbitrarily large in a given dimension, and may fail to be lower semi-continuous.
 
\begin{expl} \label{ex:mld^K depend on coef set}
Let $1<a<b$ be coprime positive integers. Consider the klt surface singularity $0\in (X=\bA^2_{xy},\Delta=\frac{1}{a}(x^a+y^b=0))$ and the weighted blowup $\pi\colon Y\to X$ with $\wt(x)=b$ and $\wt(y)=a$. The exceptional divisor $E$ is a Koll\'ar component and by adjunction one can check that $\Diff_E(\Delta_Y)=(1-\frac{1}{a})P+(1-\frac{1}{b})Q+\frac{1}{a}R$ where $P$ and $Q$ are the two singular points of $Y$ and $R\neq P,Q$ is a smooth point. It is then not hard to see that the alpha invariant $\alpha(E,\Diff_E(\Delta_Y))=1$ and hence $E$ is the only Koll\'ar component over $0\in (X,\Delta)$ by \cite{Pro-plt-blowup}*{Theorem 4.3}. It follows that 
\[
\mldk(0,X,\Delta)=A_{X,\Delta}(E)=a.
\]
By \cite{LX-stability-kc}*{Theorem 1.3}, we also have $\hvol(x,X,\Delta)=A_{X,\Delta}(E)^2\cdot (-(E^2))=\frac{a}{b}$. Thus unlike the boundedness conjecture for mld, the upper bound on $\mldk$ in Conjecture \ref{conj:bdd for mldK, vol>epsilon}, if exists, necessarily depends on the coefficient set $I$. Note also that $\mldk(x,X,\Delta)\le 2$ when $x\neq 0$. Hence the example also shows that $\mldk$ is not lower semi-continuous. 
\end{expl}

When $X$ is a surface and $\Delta=0$, one can check that $\mldk(x,X)$ coincides with the usual minimal log discrepancy $\mld(x,X)$ (see \cite{LX-mld=mldK-surface}). Our next example shows that starting in dimension three $\mldk(x,X,\Delta)$ differs from $\mld(x,X,\Delta)$, even when $\Delta=0$.

\begin{expl} 
Let $n\ge 3$ and consider the canonical hypersurface singularity
\[
0\in (x_0^{n+1}+x_1^n+x_2^n\cdots+x_n^n=0)\subseteq \bC^{n+1}.
\]
It can be resolved by the ordinary blow up, and the corresponding exceptional divisor $E$ is the only divisor that computes the mld. But $E$ is not a Koll\'ar component: it has a singular point of multiplicity $n$, in particular it's not klt.
\end{expl}

Although Example \ref{ex:mld^K depend on coef set} shows that $\mldk$ does not satisfy lower semi-continuity, we may still ask if it satisfies the other conjectural property of the mld, i.e. the ascending chain condition (ACC).

\begin{que}
Let $n\in\bN^*$ and let $I\subseteq [0,1]$ be a DCC set. Does the set
\[
\{\mldk(x,X,\Delta)\,|\,\dim X=n,\,\Coef(\Delta)\in I,\mbox{ and } x\in (X,\Delta)\mbox{ is klt }\}
\]
satisfy the ACC?
\end{que}

Recall that the Weil index of a Fano variety $X$ is the largest integer $q$ such that $-K_X\sim_\bQ qA$ for some Weil divisor $A$ on $X$. A Fano variety $X$ is said to be weakly special if $(X,D)$ is log canonical for every effective $\bQ$-divisor $D\sim_\bQ -K_X$. If Conjecture \ref{conj:bdd for mldK} were true, it would imply the following conjecture; otherwise the orbifold cone construction will produce counterexamples to Conjecture \ref{conj:bdd for mldK}.

\begin{conj}
Let $n\in\bN^*$. Then there exists some constant $N>0$ depending on $n$ such that the Weil index of any $n$-dimensional weakly special Fano variety is at most $N$.
\end{conj}

We suspect that the above conjecture may hold even for K-semistable Fano varieties. Our main intuition comes from toric examples: while a weighted projective space can have arbitrarily large Weil index, it is K-semistable only when it's $\bP^n$. If this is true, it may provide strong evidence towards Shokurov's boundedness conjecture for mld, since by \cite{XZ-uniqueness} every klt singularity admits a unique K-semistable valuation.

\bibliography{ref}

\end{document}